\documentclass{article}
\usepackage{graphicx} 
\usepackage[utf8]{inputenc}
\usepackage{enumitem,amssymb,url}
\usepackage{amsmath,amsxtra,amssymb,amsfonts,latexsym, mathrsfs, dsfont,bm,enumitem, amsthm, mathtools}
\usepackage{enumitem}
\usepackage{color,soul}
\usepackage{cite}
\usepackage{hyperref}
\usepackage{geometry}[margin=1in]

\title{A Whitney Extension Problem for Manifolds}
\author{Kevin O'Neill}

\newtheorem{theorem}{Theorem}[section]

\newtheorem{problem}[theorem]{Problem}

\newtheorem{corollary}[theorem]{Corollary}
\newtheorem{lemma}[theorem]{Lemma}

\theoremstyle{definition}
\newtheorem{definition}{Definition}[section]

\newtheorem{remark}[definition]{Remark}

\theoremstyle{definition}

\newcommand{\N}{\mathbb{N}}
\newcommand{\Z}{\mathbb{Z}}
\newcommand{\Q}{\mathbb{Q}}
\newcommand{\R}{\mathbb{R}}

\newcommand{\dist}{\text{dist}}

\newcommand{\p}{\mathcal{P}}

\renewcommand{\H}{\mathcal{H}}
\renewcommand{\d}{\partial}
\newcommand{\da}{\partial^\alpha}

\newcommand{\cyl}{\mathbf{cyl}}

\renewcommand{\P}{\mathcal{P}}

\newcommand{\kb}{{\bar{k}}}
\newcommand{\eps}{\epsilon}
\newcommand{\one}{\mathbf{1}}

\DeclareMathOperator{\Col}{Col}
\DeclareMathOperator{\Gr}{Gr}
\DeclareMathOperator{\Graff}{Graff}

\newcommand{\M}{\mathcal{M}}

\begin{document}

\maketitle

\begin{abstract}
    The purpose of this paper is to address a manifold-based version of Whitney's extension problem: Given a compact set $E\subset\R^n$, how can we tell if there exists a $d$-dimensional, $C^m$-smooth manifold $\mathcal{M}\supset E$? We provide an answer for compact manifolds with boundary in terms of a Glaeser refinement much like that used in the solution of the classical Whitney extension problem and a topological condition. This condition is the existence of a continuous selection for Grassmannian-valued functions, meant to reflect the collection of possible tangent spaces. We demonstrate the necessity of this condition in general and its non-redundancy in an example, while also showing it need not be checked when $d=1$.
\end{abstract}

\section{Introduction}

Let $C^m(\R^n)$ denote the set of $m$-times continuously differentiable functions from $\R^n$ to $\R$.

The Whitney extension problem \textemdash first posed by Whitney in 1934 \cite{W34-2} \textemdash asks the following:

\begin{problem}[Whitney Extension Problem]\label{prob:WEP}
    Given integers $m,n\ge1$, a compact subset $E\subset \R^n$, and a function $f:E\to\R$, how can we tell if there exists $F\in C^m(\R^n)$ such that $F(x)=f(x)$ for all $x\in E$?
\end{problem}

If such an $F$ exists, we call it an \textit{extension} of $f$.

The case of $n=1$ was solved by Whitney \cite{W34-2}, before Glaeser solved the $m=1$ case \cite{G58}. Fefferman \cite{F06} (see also \cite{F05-J},\cite{F05-Sh}) solved the general case following notable progress by Brudnyi-Shvartsman (see \cite{BS01}, for instance). We provide a summary of Fefferman's solution here before providing specifics in Section \ref{sec:Fefferman solution}.

Let $\P^m(\R^n)$ denote the vector space of real-valued polynomials in $n$ variables and of degree at most $m$. For each $x\in E$, define
\begin{equation}\label{eq:define usual bundle}
    H(x)=\{P\in \P^m(\R^n):P(x)=f(x)\}
\end{equation}
If $f$ has an extension $F$, then clearly $J_xF\in H(x)$ for all $x\in E$, where $J_xF$ is the $m$-th degree Taylor polynomial of $F$ at $x$. (We omit the dependence on $m$ in our notation since $m$ will be fixed.)

For each $x\in E$, determine a new set of polynomials $\tilde{H}(x)$ consisting of the elements of $H(x)$ which are compatible with elements of $H(y)$ according to Taylor's theorem for all $y$ near $x$. The collection $(\tilde{H}(x))_{x\in E}$ is called the \textit{Glaeser refinement} of $(H(x))_{x\in E}$. A key property of Glaeser refinement is that it is computable using only linear algebra and limits. (See Section \ref{sec:Fefferman solution}, also \cite{FI20-book}.) The Glaeser refinement is repeated a finite number of times, terminating in a set $H_*(x)$ for each $x\in E$. The main result of \cite{F06} is that $f$ has an extension if and only if $H_*(x)$ is nonempty for all $x\in E$. We will make use of a vector-valued version of this theorem established in\cite{FL14}.

The purpose of this paper is to address the following variant of the Whitney extension problem, on the topic of manifolds:

\begin{problem}\label{prob:manifold WEP}
    Let $d,m,n\ge1$ be integers satisfying $d<n$. Given a compact subset $E\subset\R^n$, how can we tell if there exists a $d$-dimensional, $C^m$-smooth manifold $\mathcal{M}$ such that $\mathcal{M}\supset E$?
\end{problem}

Problem \ref{prob:manifold WEP} appears to have first been formally stated in the collection of Open Problems connected with the Second Workshop on Whitney Problems in 2009 \cite{OpenProbsList}, although it was left open to determine an exact formulation of the problem. A similar version was also stated in \cite{FI20-book}.

Related questions have been answered in prior work, largely in discrete settings. In \cite{fefferman2016testing}, an algorithm is developed to find a manifold which nearly optimally interpolates a finite set of points. \cite{fefferman2022fitting} also constructs an interpolating manifold under additional assumptions with an improvement in runtime. For work on an intrinsic version of the problem in which one attempts to construct a Riemannian manifold matching a collection of distances, see \cite{Intrinsic1} and \cite{Intrinsic2}. However, the current paper appears to be the first to address the case of infinite $E$. 

For positive integers $d_1,d_2\ge1$, let $\P^m(\R^{d_1},\R^{d_2})$ denote the vector space of $\R^{d_2}$-valued polynomials in $d_1$ variables and of degree at most $m$ and $C^m(\R^{d_1},\R^{d_2})$ denote the set of $m$-times continuously differentiable functions from $\R^{d_1}$ to $\R^{d_2}$. To determine the existence of a $d$-dimensional, $C^m$-smooth manifold containing a set $E$, we assign to each $x\in E$ a subset $H(x)\subset O(n)\times \P^m(\R^d,\R^{n-d})$. As a $d$-dimensional manifold in $\R^n$ is locally the graph of a $C^m(\R^d,\R^{n-d})$ function in some coordinate system, we use the orthogonal group, $O(n)$ to determine an orthonormal coordinate system for $\R^n$. The element $(Q,P)\in O(n)\times \P^m(\R^d,\R^{n-d})$ is meant to represent a possible Taylor polynomial $P$ of such a function in the coordinate system determined by $Q$. An advantage of this approach is that it does not require one to compute any change of coordinates formulas for jets, simply acknowledge that they exist (see Lemma \ref{lemma:jet realization well-defined}).

The sets $H(x)$ are constructed as follows: For each $Q\in O(n)$, determine a coordinate system for $\R^n$. Then view $x$ as plot of a point $(y,f(y))$ in that coordinate system. Following \eqref{eq:define usual bundle}, include in $H(x)$ all pairs $(Q,P)$ where $P\in \P^m(\R^d,\R^{n-d})$ satisfies $P(y)=f(y)$. Do this for all $Q\in O(n)$.

We define a notion of Glaeser refinement for this collection $(H(x))_{x\in E}$ and show in Lemma \ref{lemma:termination} that this refinement terminates in a finite number of steps, leaving us with a set $H_*(x)$ for each $x\in E$. Our main result, Theorem \ref{thm:main}, states that there is a $d$-dimensional, $C^m$-smooth, compact manifold with boundary $\mathcal{M}\supset E$ if and only if $H_*(x)$ is nonempty for all $x\in E$ and a certain topological condition is satisfied. We will refer to such an $\M$ as a \textit{manifold extension} of $E$, specifying the particular qualities of $\M$ (possession of boundary, etc.) when required. Theorem \ref{thm:main} is the result of recognizing and addressing a novel topological obstacle and developing a technical framework for connecting prior work on the Whitney extension problem to the study of manifolds.

We now take a moment to motivate this topological condition with an example. Consider the subset of $\R^4$
\begin{equation*}
    E_1=\{(x_1,x_2,0,0):x_1^2+x_2^2\le 1\} \cup \{(\cos t,\sin t, s\cos (t/2),s\sin (t/2):0\le t\le 2\pi, 0\le s\le 1\}.
\end{equation*}
We ask whether there exists a 3-dimensional, $C^1$ manifold containing $E_1$. To begin, let us determine, for each point $p\in E_1$, whether there exists $\delta>0$ and a 3-dimensional, $C^1$ manifold $\mathcal{M}_p$ such that
\begin{equation*}
    \mathcal{M}_p\cap B(p,\delta)\supset E_1\cap B(p,\delta),
\end{equation*}
where $B(x,r)$ refers to the open ball of radius $r>0$ centered at $x\in \R^n$. We will also take note of some constraints on $T_p\mathcal{M}_p$, where $T_x\mathcal{M}$ refers to the tangent space of the manifold $\mathcal{M}$ at the point $x\in \mathcal{M}$.

If $p$ is of the form $(\cos t_0,\sin t_0,s_0\cos (t_0/2),s_0\sin (t_0/2))$ for fixed $0\le t_0\le 2\pi$ and $0<s_0\le 1$, then such an $\M_p$ exists, with 
\begin{equation*}
    \mathcal{M}_p=\{(r\cos t,r\sin t, s\cos (t/2),s\sin (t/2)):t_0-\delta\le t\le t_0+\delta, 0<s_0-\delta\le s\le s_0+\delta,1-\delta<r<1+\delta\}
\end{equation*}
serving as one possible example.

For any $p$ of the form $(\cos t_0,\sin t_0,0,0)$ we have the manifold
\begin{equation*}
    \mathcal{M}_p=\{(r\cos t,r\sin t, s\cos (t/2),s\sin (t/2)):t_0-\delta\le t\le t_0+\delta, -\delta<s<\delta,1-\delta<r<1+\delta\}.
\end{equation*}
The tangent space $T_p\mathcal{M}_p$ is uniquely determined since
\begin{equation*}
    p+c(1,0,0,0)\in E_1
\end{equation*}
for arbitrarily small values of $c$ and $E_1\cap B(p,\delta)$ contains a two-dimensional subset of $\{(x,y,0,0):x,y\in\R\}$ for all $\delta>0$.

Thus,
\begin{equation*}
    T_p\mathcal{M}_p=p+Span\{(1,0,0,0),(0,1,0,0),(0,0,\cos (t_0/2), \sin (t_0/2))\}
\end{equation*}
for any choice of $\mathcal{M}_p$.

The remaining points of $E_1$ are of the form $p=(x_1,x_2,0,0)$ for $x_1^2+x_2^2<1$. For these points, we may choose
\begin{equation}\label{eq:M_p tangent space}
    p+Span\{(1,0,0,0),(0,1,0,0),(0,0,\cos (\theta/2),\sin (\theta/2))\}
\end{equation}
to be both $\mathcal{M}_p$ and $T_p\mathcal{M}_p$, where $\theta\in[0,2\pi]$ is allowed to be arbitrary. Furthermore, whatever the choice of $\M_p$, $T_p\M_p$ must be of the form \eqref{eq:M_p tangent space} for some $\theta\in[0,2\pi]$ since
\begin{equation*}
    p+c(1,0,0,0),p+c(0,1,0,0)\in E_1
\end{equation*}
for arbitrarily small values of $c$ and a vector space is uniquely determined by an orthonormal basis.

We have shown two things thus far, that for each $p\in E_1$, the intersection of $E_1$ with a neighborhood of $p$ lies in a 3-dimensional, $C^1$ manifold, and that in some cases the tangent space of such a manifold must satisfy certain constraints. 

If there exists a single 3-dimensional, $C^1$ manifold $\mathcal{M}\supset E_1$, then its tangent space must vary continuously. In this example, this means, for each $p=(\cos t,\sin t,0,0)$, we must choose $\theta(p)\in[0,2\pi]$ such that
\begin{equation*}
    (\cos \theta(p),\sin\theta(p))=\pm (\cos (t/2),\sin (t/2)) 
\end{equation*}
and
\begin{equation*}
    \theta:\{(x_1,x_2:x_1^2+x_2^2\le 1\}\to \Gr(2,1) \text{ is continuous},
\end{equation*}
where $\Gr(n,d)$ denotes the Grassmannian of $d$-dimensional planes through the origin in $\R^n$. Here, we identify $\pm (\cos (t/2),\sin (t/2))$ with the line through the origin and $\pm (\cos (t/2),\sin (t/2))$.

However, this implies no such $\M$ exists as the restriction of $\theta(\cdot)$ to $\{(x_1,x_2,0,0):x_1^2+x_2^2=1\}$ (a space homeomorphic to $S^1$) is a non-trivial loop in $\Gr(2,1)$; a continuous extension of $\theta$ to $\{(x_1,x_2,0,0):x_1^2+x_2^2\le 1\}$ would imply a contraction of the loop.

This example may be extended to the case of fitting $d$-dimensional manifolds in $\R^n$ whenever $3\le d<n$ by forcing the tangent spaces at $(\cos t,\sin t,0,\ldots,0)$ to be $\R^2+V(t)$, where $V(t)$ is a nontrivial loop in $\Gr(n-2,d-2)$. A similar argument applies as such Grassmannians are not simply connected.

Interesting phenomena may occur with regards to which particular loops in the Grassmannian are trivial or nontrivial, as seen in the following variation of the above example, due to Shmuel Weinberger. Let $E_1'=\iota(E_1)\subset\R^5$ where $\iota:(x_1,x_2,x_3,x_4)\mapsto (x_1,x_2,x_3,x_4,x_5)$. Then, there is no 3-dimensional, $C^1$-manifold containing $E_1'$, as the loop defined by
\begin{equation}\label{eq:single loop}
    t\mapsto (\cos(t/2),\sin(t/2),0)
\end{equation}
remains a nontrivial loop in $\Gr(3,1)$, which is homeomorphic to $\mathbb{RP}^2$, the two-dimensional real projective space.

The story is different for
\begin{equation*}
    E_2=\{(x_1,x_2,0,0):x_1^2+x_2^2\le 1\} \cup \{(\cos t,\sin t, s\cos t,s\sin t:0\le t\le 2\pi, 0\le s\le 1\}
\end{equation*}
in place of $E_1$. There is no 3-dimensional $C^1$ manifold containing $E_2\subset \R^4$, as
\begin{equation*}
    t\mapsto (\cos t,\sin t)
\end{equation*}
is a nontrivial loop in $\Gr(2,1)$. Nevertheless, there \textit{is} a 3-dimensional, $C^1$ manifold with boundary containing $\iota(E_2)\in \R^5$, as the loop
\begin{equation*}
    t\mapsto (\cos t,\sin t,0)
\end{equation*}
is the loop in \eqref{eq:single loop} traced out twice and the fundamental group of $\mathbb{RP}^2$ is $\Z_2$.

The above examples are in contrast to the usual notion of extension (or interpolation) of smooth functions in Euclidean space, where local extensions (or interpolants) may be pasted together via a partition of unity to form a global extension (or interpolant) without worry of topological obstacle.

We are now ready to somewhat formally state the topological condition in Theorem \ref{thm:main}. At each $x\in E$, let $\hat{H}(x)\subset \Gr(n,d)$ such that $p+\hat{H}(x)$ is the set of possible tangent planes determined by the collections of coordinate frames and Taylor polynomials found in $H_*(x)$. The topological condition is that there exists a continuous $f:E\to \Gr(n,d)$ such that $f(x)\in \hat{H}(x)$ for all $x\in E$. 

The manifold with boundary guaranteed by our main result need not be connected. In the construction, we first paste together of the graphs of $C^m$ functions on open sets, then obtain a compact manifold with boundary by intersecting the result with a smooth hypersurface. To determine the existence of a connected manifold will likely require stronger analysis of global topological issues.

A theme in the solution to the Whitney extension problem is that the answer (yes or no extension) is computable with only linear algebra and the ability to take a limit. The present form of Glaeser refinement and the computation of $\hat{H}(x)$ have this same property, though we do not provide a complete solution regarding the topological problem.

In Section \ref{sec:Fefferman solution}, we provide background on some tools necessary for our analysis as well as the classical Whitney extension problem for $C^m(\R^n,\R^d)$, including a formal definition of the aforementioned Glaeser refinement and the solution given in Theorem \ref{thm:v-valued}. The interested reader is encouraged to consult \cite{F09-Int} and \cite{FI20-book} for more on the topic.

Section \ref{sec:mfld GR} begins with a formal discussion of our version of Glaeser refinement for Problem \ref{prob:manifold WEP}, including a proof that it terminates after a controlled number of iterations. It ends with a statement of our main theorem.

In Section \ref{sec:technical results}, we prove a few technical results that will be useful later on. As a corollary, we deduce that in order to determine the existence of a manifold extension, we need only consider a finite set of $Q\in O(n)$, reducing our computations to, essentially, a finite number of the usual Glaeser refinements (see Theorem \ref{thm:finite number Q's}).

The proof of Theorem \ref{thm:main} begins in Section \ref{sec:local interpolants}, where we use elements of $H_*(x)$ to construct local interpolants (or local manifolds) by direction application of Theorem \ref{thm:v-valued}. These interpolants are carefully chosen to be compatible using our topological condition before being pasted together into a manifold in Section \ref{sec:pasting}. To smooth out the boundary, we intersect the manifold with a generic $C^\infty$ surface. The pasting process is adapted from a similar method in \cite{fefferman2016testing}. This will complete the proof of Theorem \ref{thm:main}.

Next, we demonstrate in Section \ref{sec:topologicla condition} that for $d=1$, one need not check the topological condition to determine the existence of a manifold extension. A key tool here is a selection theorem of Michael \cite{michael1956continuous} which provides sufficient conditions for one to be able to continuously choose a function value from a collection of subsets when the codomain provides topological obstacles. While the $d=1$ case may not require this machinery, the approach will shed light on why the topological condition need be checked in higher dimensions and assist in identifying obstacles for selection in the case $d\ge2$.

As one last remark, we note that in the solution to the classical Whitney extension problem, a finiteness principle solving a discrete version of the problem was first proven, then used to address the continuous case. Here, we jump ahead to the continuous version of the problem since we are able to use the solution to the classical Whitney extension problem as a black box. It may be interesting to see what discrete versions of our results may hold.

\section{Background Material}\label{sec:Fefferman solution}

\subsection{Basic Notation}

We use $C$ and $c$ to denote arbitrary large and small constants, respectively. A constant $C=C(S)$ depends only on the parameters $S$. $C_1,c_1,\ldots$ will denote recurring constants.

Let $N,D\in\N$. If $F:\R^N\to\R^D$ is differentiable, we write its gradient, $\nabla F:\R^N\to\R^{D\times N}$ as the matrix-valued function satisfying
\begin{equation*}
    F(y)=F(x)+\nabla F(x)(y-x)+o(|y-x|),
\end{equation*}
where $(x-y)$ is viewed as a column vector in $\R^N$ in the above matrix-vector multiplication. $\nabla^kF$ is the order $k+1$ tensor consisting of the $k$-th order derivatives of $F$. If $F$ is $C^m$ for all $m\ge0$, then we say $F$ is $C^\infty$.

For a particular $k$-th order derivative, we use the notation $\da F$, where $\alpha=(\alpha_1,\ldots,\alpha_N)$ is a multi-index such that $\alpha_1+\ldots+\alpha_N=k$. Note that $\da F$ is an $\R^D$-valued function.

On $\Gr(n,d)$, introduce the distance function
\begin{equation}\label{eq:Grass dist}
    \dist(W,W')=\|\Pi_W-\Pi_{W'}\| \text{ for }W,W'\in \Gr(n,d),
\end{equation}
where $\Pi_W$ refers to the orthogonal projection of $\R^n$ onto $W$ and $\|\cdot\|$ is any norm on the space of linear operators on $\R^n$.

\begin{lemma}\label{lemma:Grass bases}
    Let $d,n\in\N$. There exist constants $C=C(d,n)$ and $\eps_0=\eps_0(d,n)>0$ such that whenever $W,W'\in \Gr(n,d)$ satisfy $\dist(W,W')\le\epsilon\le \eps_0$, there exists $Q\in O(n)$ such that $QW'=W$ and $\|Q-I_n\|\le C\epsilon$.
\end{lemma}

\begin{proof}
    By a rotation in $\R^n$, we may assume without loss of generality that $W'=Span\{e_1,\ldots,e_k\}$, where $e_1,\ldots,e_n$ is the standard basis for $\R^n$. Define
    \begin{equation*}
        v_i=\Pi_W(e_i), 1\le i\le d.
    \end{equation*}
    By hypothesis,
    \begin{equation}\label{eq:v_i close to e_i}
        \|v_i-e_i\|=\|\Pi_W(e_i)-\Pi_{W'}(e_i)\|\le \dist(W,W')\|e_i\|\le\epsilon
    \end{equation}
    for $1\le i\le d$. We run the Gram-Schmidt process on $\{v_1,\ldots,v_d\}$ to create an orthonormal basis $\{w_1,\ldots,w_d\}$ for $W$.

    First, take $u_1=v_1$. By \eqref{eq:v_i close to e_i}, $\|v_1-e_1\|\le \eps$. Now suppose that for all $1\le i\le k$ (where $k<d$), there exists $C$ such that
    \begin{equation}\label{eq:close basis vector}
        \|u_i-e_i\|\le C\eps.
    \end{equation}
    Note that \eqref{eq:close basis vector} implies
    \begin{equation}\label{eq:u_i norm near 1}
        |\|u_i\|-1|\le C\eps.
    \end{equation}
    The Gram-Schmidt process gives
    \begin{equation*}
        u_{k+1}=v_{k+1}-\sum_{i=1}^k\frac{\langle v_{k+1},u_i\rangle}{\langle u_i,u_i\rangle}u_i,
    \end{equation*}
    and one sees by \eqref{eq:v_i close to e_i} that
    \begin{align*}
        \|u_{k+1}-e_{k+1}\|&\le \|u_{k+1}-v_{k+1}\|+\|v_{k+1}-e_{k+1}\|\\
        &\le kC\eps+C\eps\le C\eps.
    \end{align*}
    Define $w_i=\frac{u_i}{\|u_i\|}$ to create an orthonormal basis and observe $\|w_i-e_i\|\le C\eps$ by \eqref{eq:close basis vector} and \eqref{eq:u_i norm near 1}.

    Repeating this process for the orthogonal complements of $W$ and $W'$, we find an orthonormal basis $\{v_{d+1},\ldots,v_n\}$ for $W^\perp$ such that
    \begin{equation}\label{eq:final basis close}
        \|v_i-e_i\|<C\eps \text{ for }d+1\le i\le n.
    \end{equation}

    Form an orthogonal matrix $Q$ with columns $v_1,\ldots, v_n$. Then $QW'=W$ and $\|Q-I_n\|\le C\eps$ by \eqref{eq:final basis close}.    
\end{proof}

Let $Q\in O(n)$ such that the columns of $Q$ are $v_1,\ldots,v_n$. (Here and in the sequel, we identify an orthogonal transformation with its matrix representation with respect to the standard ordered basis of $\R^n$.) Define $\Pi:O(n)\to \Gr(n,d)$ by
\begin{equation}\label{eq:define Pi}
    \Pi(Q)=Span\{v_1,\ldots,v_d\}.
\end{equation}
The dimension $d$ will be fixed, thus given by context.

The affine Grassmannian $\Graff(n,d)$ consists of all $d$-planes in $\R^n$ (containing the origin or not). It has a natural topology under which i) the map $\mathcal{M}\to \Graff(n,d)$ by $x\mapsto T_x\mathcal{M}$ is continuous for any $d$-dimensional $C^1$ manifold $\mathcal{M}\subset\R^n$ and ii) the natural projection $p:\Graff(n,d)\to \Gr(n,d)$ by
\begin{equation}\label{eq:natural projection}
    p(x+W)=W
\end{equation}
for $x\in\R^n$, $W\in \Gr(n,d)$ is continuous.

\subsection{Manifolds and Transversality}

Given $x\in\R^n$, $\delta>0$, we let $B(x,\delta)$ denote the open ball of radius $\delta$ with center $x$. If $A\subset\R^n$ is closed, define $\dist(x,A)=\inf_{y\in A} |x-y|$.

We say $\M\subset\R^n$ is a $d$-dimensional, $C^m$ manifold if for all $x\in \M$, there exists $\delta>0$ such that $B(x,\delta)\cap \M$ is the graph of a $C^m(\R^d,\R^{n-d})$ function in some coordinate system, or equivalently, if there exists $\delta>0$ such that $B(x,\delta)\cap \M=F(U)$ for some open $U\subset \R^d$ and injective $C^m$ function $F:U\to\M$. A manifold need not be closed as a subset of $\R^n$.

We say $\M\subset\R^n$ is a $d$-dimensional, $C^m$ manifold with boundary if for all $x\in \M$, there exists $\delta>0$ such that $\M=F(U)$ for some injective $C^m$ function $F:U\to\M$ where $U$ is either an open subset of $\R^d$ or a (relatively) open subset of $\R^d_+:=\{(x_1,\ldots,x_d)\in\R^d:x_d\ge0\}$. If $\M$ is a closed subset of $\R^n$, we refer to it as a compact manifold with boundary.

If $\mathcal{M}$ is a $C^1$ manifold and $x\in \M$, we denote the tangent space of $\M$ at $x$ by $T_x\M$. If $\M\subset \R^N\times \R^D$ is the graph of $F\in C^1(\R^N,\R^D)$, then
\begin{equation}\label{eq:tangent space of graph}
    T_{(x,F(x))}\M=(x,F(x))+\Col\left(\begin{bmatrix}
    I_N\\
    \nabla F(x)
\end{bmatrix}\right),
\end{equation}
where $I_N$ is the $N\times N$ identity matrix and $\Col(A)$ refers to the column space of a matrix $A$.

We say two submanifolds $\mathcal{M}$ and $\mathcal{N}$ of $\R^n$ \textit{intersect transversely} if for all $z\in \M\cap\mathcal{N}$, $T_z\M+T_z\mathcal{N}=\R^n$, where the addition is the standard sum of vector subspaces.

The following is a standard result on the transverse intersection of manifolds (see Theorem 6.30b) in \cite{lee2012smooth}, for instance, for a slightly more general result).

\begin{theorem}\label{thm:transverse intersection is manifold}
    Suppose $\mathcal{M}$ and $\mathcal{N}$ are $C^\infty$ submanifolds of $\R^n$. If $\mathcal{M}$ and $\mathcal{N}$ intersect transversely, then $\mathcal{M}\cap \mathcal{N}$ is a $C^\infty$ manifold.
\end{theorem}

The following is a standard application of Thom's transversality theorem \cite{thom} (see Example 3.1 of \cite{greenblatt2015introduction}):

\begin{theorem}\label{thm:a.e. is transverse}
    Let $\mathcal{M}$ and $\mathcal{N}$ be $C^\infty$ submanifolds of $\R^n$. Then, for almost every $v\in\R^n$, $v+\M$ intersects $\mathcal{N}$ transversely.
\end{theorem}

\subsection{Solution to the Classical Whitney Extension Problem}

Fix integers $D,m,N\ge 1$. Let $\p^m(\R^N,\R^D)=\p^m(\R^N)\oplus\ldots\oplus\p^m(\R^N)$ denote the vector space of $\R^D$ valued-valued polynomials in $N$ real variables and of degree at most $m$. If $F\in C^m(\R^N,\R^D)$, we denote the $m$-th degree Taylor polynomial of $F$ by $J_xF$. Note that $J_xF\in \p^m(\R^N,\R^D)$. We call $J_xF$ the \textit{jet} of $F$ at $x$.

Given a compact set $E\subset\R^N$, we define a \textit{bundle} $\H=(H(x))_{x\in E}$ to be a collection of affine subspaces $H(x)\subset \p^m(\R^N,\R^D)$ (called \textit{fibers}). We say $F\in C^m(\R^N,\R^D)$ is a \textit{section} of $\H$ if $J_xF\in H(x)$ for all $x\in E$.

If $(H(x))_{x\in E}$ is a bundle, we define a second bundle, called the Glaeser refinement of $(H(x))_{x\in E}$ and denoted $\tilde{\H}=(\tilde{H}(x)))_{x\in E}$, as follows:

For all $x_0\in E$, $\tilde{H}(x_0)$ consists of the elements $P_0=(P_0^1,\ldots,P_0^D)\in H(x_0)$ such that for all $\eps>0$, there exists $\delta>0$ such that for all $x_1,\ldots,x_{\kb}\in E\cap B(x_0,\delta)$ there exist $P_i=(P_i^1,\ldots,P_i^D)\in H(x_i)$ ($1\le i\le \kb)$ satisfying
\begin{equation*}
    |\da (P_i^l-P_j^l)(x_i)|\le \eps|x_i-x_j|^{m-|\alpha|}
\end{equation*}
for all $0\le i,j\le \kb, 1\le l\le D$, and $|\alpha|\le m$. Here, $\kb=\kb(m,N,D)$ is a positive integer. (It is elementary to verify that each fiber $\tilde{H}(x)$ is an affine space so $\tilde{\H}$ is indeed a bundle by definition.) Let $(H_1(x))_{x\in E},(H_2(x))_{x\in E},\ldots$ denote the iterated Glaeser refinements of $(H_0(x))_{x\in E}=(H(x))_{x\in E}$.

An important property of Glaeser refinement is that it is computable with only linear algebra and the ability to take a limit. To see this, fix $x_0\in E, P_0\in H(x_0)$ and for $x_1,\ldots,x_\kb\in E$ define
\begin{equation*}
    \mathcal{Q}(P_1,\ldots,P_{\kb};x_1,\ldots,x_{\kb})=\sum_{\substack{0\le i,j\le \kb\\ x_i\neq x_j}}\sum_{|\alpha|\le m} \sum_{1\le l\le D}\left[\frac{\da(P^l_i-P^l_j)(x_i)}{|x_i-x_j|^{m-|\alpha|}}\right]^2.
\end{equation*}

For fixed $x_1,\ldots,x_\kb$, $\mathcal{Q}(P_1,\ldots,P_{\kb};x_1,\ldots,x_{\kb})$ is a positive semidefinite quadratic form. Since each $H(x_i)$ is an affine space,
\begin{equation*}
    \mathcal{Q}_{\min}(x_1,\ldots,x_{\kb})=\min\{\mathcal{Q}(P_1,\ldots,P_{\kb};x_1,\ldots,x_{\kb}):P_i\in H(x_i), 1\le i\le \kb\}
\end{equation*}
may be computed with linear algebra. We see that $P_0\in \tilde{H}(x_0)$ if and only if $\mathcal{Q}_{\min}(x_1,\ldots,x_{\kb})\to 0$ as $(x_1,\ldots,x_{\kb})\to (x_0,\ldots,x_0)$ in $E^{\kb}$.

We note that $(\tilde{H}(x))_{x\in E}$ has a section if and only if $(H(x))_{x\in E}$ has a section. If $(H(x))_{x\in E}$ has a section $F$, then one may take $P_i=J_{x_i}F$ in the definition of Glaeser refinement to show $J_xF\in \tilde{H}(x)$ for all $x\in E$. The reverse direction is trivial since $\tilde{H}(x)\subset H(x)$ for all $x\in E$.

We say $(H(x))_{x\in E}$ is \textit{Glaeser-stable} if it is equal to its Glaeser refinement. The following lemma was proven in \cite{FL14} and adapted from the corresponding lemma in \cite{F06}, which in turn was adapted from \cite{G58} and \cite{BMP03}.

\begin{lemma}\label{lemma:old termination}
    Let $E\subset\R^N$ be compact and $(H(x))_{x\in E}$ be a bundle. Then, there exists $l_*=l_*(m,N,D)$ such that $H_k(x)=H_{l_*}(x)$ for all $k\ge l_*$. In particular, $(H_{l_*}(x))_{x\in E}$ is Glaeser-stable.
\end{lemma}

Given $f:E\to \R^D$, define
\begin{equation*}
    H(x)=\{P\in\p^m(\R^N,\R^D):P(x)=f(x)\}.
\end{equation*}

The question of whether there exists $F\in C^m(\R^N,\R^D)$ such that $F(x)=f(x)$ for all $x\in E$ is equivalent to whether $(H(x))_{x\in E}$ has a section. By Lemma \ref{lemma:old termination}, this is the case if and only if $(H_{l_*}(x))_{x\in E}$ has a section. This question is answered by the following theorem from \cite{FL14}.

\begin{theorem}\label{thm:v-valued}
    Let $E\subset\R^N$ and $(H(x))_{x\in E}$ be a Glaeser-stable bundle such that each fiber $H(x)$ is nonempty. Then $(H(x))_{x\in E}$ has a section.

    Furthermore, given $x_0\in E$ and $P\in H(x_0)$, one may choose $F$ such that $J_{x_0}F=P$.
\end{theorem}

The $D=1$ case was first proven in \cite{F06} and was used directly to establish Theorem \ref{thm:v-valued} via the `gradient trick.'

Clearly, if any $H(x)$ is empty, then $(H(x))_{x\in E}$ cannot have a section. Thus, Theorem \ref{thm:v-valued} provides a complete solution to Problem \ref{prob:WEP}.

\section{Glaeser Refinement for Manifolds}\label{sec:mfld GR}

Fix integers $d,m,n\ge1$ satisfying $d<n$ throughout the remainder of the paper.

Given $Q\in O(n)$, define a coordinate system centered at the origin with axes $(q_1,\ldots.q_n)$, where $q_1,\ldots,q_n$ are the columns of $Q$. We split the coordinates into the first $d$ (viewed as independent variables) and the last $n-d$ (viewed as dependent variables), writing, for $x\in \R^n$,
\begin{equation*}
    x_Q=(q_1\cdot x,\ldots,q_d\cdot x)
\end{equation*}
and
\begin{equation*}
    x^\perp_Q=(q_{d+1}\cdot x,\ldots,q_n\cdot x).
\end{equation*}

That is,
\begin{equation}\label{eq:def coords}
    \begin{bmatrix}
        x_Q\\x_Q^\perp
    \end{bmatrix}=Q^Tx,
\end{equation}
where $x_Q\in\R^d,x_Q^\perp\in\R^{n-d}$, and $A^T$ refers to the transpose of the matrix $A$.

The original vector $x$ may be recovered from $x_Q$ and $x^\perp_Q$ by the formula
\begin{equation*}
    x=\begin{bmatrix} q_1 & \ldots & q_d \end{bmatrix} x_Q + \begin{bmatrix} q_{d+1} & \ldots & q_n \end{bmatrix} x^\perp_Q=Q\begin{bmatrix}
        x_Q\\ x_Q^\perp
    \end{bmatrix},
\end{equation*}
where $x_Q$ and $x^\perp_Q$ are interpreted as column vectors for matrix multiplication. We call $x_Q$ and $x_Q^\perp$ the $Q$-coordinates of $x$ and the resulting coordinate system the $Q$-coordinate system. The definition of the $Q$-coordinate system depends on $d$, though we hide this dependence in our notation as $d$ is fixed.

Given $Q\in O(n)$ and $F\in C^m(\R^d,\R^{n-d})$, let $G_Q(F)$ be the graph of $F$ in the $Q$-coordinate system, that is,
\begin{equation*}
    G_Q(F):=\{Q[(v,F(v))]:v\in \R^{n-d}\}.
\end{equation*}
Note that $G_Q(F)$ is a $C^m$ manifold defined so that $x\in G_Q(F)$ if and only if $F(x_Q)=x_Q^\perp)$.

At each point $x\in E$, let
\begin{equation}\label{eq:def mfld fiber}
    H(x)\subset O(n)\times \P^m(\R^d,\R^{n-d})
\end{equation}
such that 
\begin{equation}\label{eq:def H^Q}
    H^Q(x):=\{P:(Q,P)\in H(x)\}
\end{equation}
is an affine space for fixed $Q\in O(n)$ and all $x\in E$. In this case, we say $(H(x))_{x\in E}$ is an \textit{M-bundle} and each $H(x)$ is a \textit{fiber} of the M-bundle at $x$. We say an M-bundle $(H(x))_{x\in E}$ is \textit{nontrivial} if $H(x)$ is nonempty for all $x\in E$. We will occasionally refer to a pair $(Q,P)$ as a \textit{jet}.

Intuitively, the idea is that given a manifold $\M\ni x$ and $Q\in O(n)$ such that $\M$ is locally the graph of a $C^m$ function $F$ in $Q$ coordinates near $x$, $(Q,J_{x_Q}F)$ should be considered `a' jet of $\M$ at $x$. The graph $G_Q(J_{x_Q}F)$ well-approximates the graph $G_Q(F)$ near $x$.

A natural choice of M-bundle is
\begin{equation}\label{eq:def H_0}
    H_0(x)=\{(Q,P)\in O(n)\times P^{m}(\R^d,\R^{n-d}): P(x_Q)=x_Q^\perp\}.
\end{equation}

The purpose of an M-bundle is to keep track of all such possible approximations. We cannot fix a single $Q$ since it is possible that a given tangent space may not be viewable as a function in $Q$ coordinates for fixed $Q$.

However, we should acknowledge that by including all $Q\in O(n)$ as possibilities for coordinate systems, we are including a lot of redundancies. First, a portion of a manifold may be viewed as the graph of a function with many different choices for independent and dependent variables. For instance, if $\M=\{(x,y,z):x=y=z\}\subset\R^3$, then $\M$ may be viewed as a function of $x$, writing $\M=G_{Q_1}(F)$, where
\begin{equation*}
    Q_1=I_3, F(t)=(t,t),
\end{equation*}
or as a function of $y$ by $\M=G_{Q_2}(F)$, or a function of $z$ by $\M=G_{Q_3}(F)$, where
\begin{equation*}
    Q_2=\begin{bmatrix}
        0&1&0\\1&0&0\\0&0&1
    \end{bmatrix}, Q_3=\begin{bmatrix}
        0&1&0\\0&0&1\\1&0&0
    \end{bmatrix}.
\end{equation*}

Second, there are natural symmetries due to changes of coordinates within the independent variables. For instance, in the case where $d=2$ and $n=3$, $G_{Q_4}(y_1+y_2^2)=G_{Q_5}(y_1^2+y_2)$, where
\begin{equation*}
    Q_4=I_3, Q_5=\begin{bmatrix}
        0&1&0\\1&0&0\\0&0&1
    \end{bmatrix}.
\end{equation*}
Similar redundancies can occur by elements of $O(n)$ which permute the last $n-d$ coordinates, rather than the first $d$.

One could avoid the redundant information found in an M-bundle by forming equivalence classes in $O(n)\times\P^m(\R^d,\R^{n-d})$ (see Lemma \ref{lemma:jet realization well-defined} and Corollary \ref{cor:only cor?} as a starting point). However, the calculations involved may be particularly difficult (see the proof of Lemma \ref{lemma:jet realization well-defined}), so we will keep the extra information and develop tools for translating from one coordinate system to another. Of particular relevance to this issue will be the notion of a consistent M-bundle, defined later in this section.

\begin{definition}
Given an M-bundle $(H(x))_{x\in E}$, we define its Glaeser refinement $(\tilde{H}(x))_{x\in E}$ as follows.

Following \eqref{eq:def H^Q}, write $H(x)=\bigcup_{Q\in O(n)}H^Q(x)$ for $x\in E$. We say $P_0\in\widetilde{H^Q}(x_0)$ if and only if
\begin{enumerate}
    \item $P_0\in H^Q(x_0)$ and
    \item For all $\epsilon>0$, there exists $\delta>0$ such that if $x_1,\ldots,x_{\kb}\in E\cap B(x_0,\delta)$, there exist $P_i\in H^Q(x_i)$ such that
    \begin{equation}\label{eq:def GR 2}
        |\da (P_i-P_j)((x_i)_Q)|\le\epsilon|(x_i)_Q-(x_j)_Q|^{m-|\alpha|}
        \text{for all }|\alpha|\le m, 0\le i,j,\le \kb.
    \end{equation}
\end{enumerate}
Lastly, let
\begin{equation*}
    \tilde{H}(x)=\bigcup_{Q\in O(n)}\widetilde{H^Q}(x).
\end{equation*}
\end{definition}

Observe that $\widetilde{H^Q}(x)=\tilde{H}^Q(x)$, where $(\tilde{H}(x))_{x\in E}$ is the Glaeser refinement of $(H(x))_{x\in E}$ and
\begin{equation*}
    \tilde{H}^Q(x)=\{P:(Q,P)\in \tilde{H}(x)\}
\end{equation*}
as in \eqref{eq:def H^Q}; we will use both notations interchangeably, referring to $(\tilde{H}^Q(x))_{x\in E}$ as the Glaeser refinement of $(H^Q(x))_{x\in E}$. We say an M-bundle is \textit{Glaeser-stable} if it is its own Glaeser refinement.

The definition of Glaeser refinement for $H^Q(x)$ is nearly equivalent to the definition of Glaeser refinement given in Section 2 for the usual type of bundle. However, for bundles we consider $x_i\in\R^d$ and $B(x_0,\delta)$ is a $d$-dimensional ball. For M-bundles, $x_i\in\R^n$ and $B(x_0,\delta)$ is an $n$-dimensional ball. The compatibility of the $P_i$ may be tested through consideration of the $(x_i)_Q\in\R^d$, though the original $x_i\in\R^n$ are still relevant.

To see this difference expressed in an example, consider the Dirichlet function
\begin{equation*}
    f(x)=\begin{cases}
        1 & x\in \Q\\
        0 & x\notin \Q
    \end{cases},
\end{equation*}
which is known to be discontinuous. In particular, there is no $C^1$ extension of $f:[0,1]\to\R$ to the entire real line. Defining $H(x)=\{f(x)+m(t-x):m\in\R\}$ for $x\in [0,1]$, one may readily verify that the Glaeser refinement of the bundle $(H(x))_{x\in E}$ has empty fibers; for each $x\in [0,1]$ there exist arbitrarily close $t$ to $x$ such that $|f(t)-f(x)|=1$. (Here, $f(x)+m(t-x)$ is viewed as a polynomial in $t$.)

Now consider the graph of the Dirichlet function: the set
\begin{equation*}
    E=\{(x,1):x\in\Q\cap[0,1]\}\cup\{(x,0):x\in[0,1]\setminus\Q\}\subset\R^2.
\end{equation*}
We define an M-bundle as in \eqref{eq:def H_0} and observe that
\begin{equation*}
    H^{I_2}((x,y))=\begin{cases}
        \{1+m(t-x):m\in\R\} & x\in \Q\\
        \{0+m(t-x):m\in\R\} & x\notin\Q
    \end{cases}
\end{equation*}
for $(x,y)\in E$.

This M-bundle remains nontrivial after Glaeser refinement because the points $(x,0)$ and $(x',1)$ are far apart in $\R^2$, even when $x$ is near $x'$. For Glaeser refinement, $H^{I_2}(x)$ need only be checked against $H^{I_2}(x')$ if both $x$ and $x'$ are rational, or if both $x$ and $x'$ are irrational. We note there is a $C^1$ manifold containing $E$, for instance, $\{(x,0):x\in(-.5,1.5)\}\cup\{(x,0):x\in(-.5,1.5)\}$.

\begin{lemma}\label{lemma:affine space refinement}
    If $(H(x))_{x\in E}$ is an M-bundle, then so is $(\tilde{H}(x))_{x\in E}$.
\end{lemma}

\begin{proof}
    The part of this statement requiring proof is that $\tilde{H}^Q(x)$ is an affine space.

    Let $x_0\in E, Q\in O(n)$ and suppose $P_0,P_0'\in \tilde{H}^Q(x_0)$. We want to show $\lambda P_0+(1-\lambda)P_0'\in \tilde{H}^Q(x_0)$ for $\lambda\in\R$.
    
    Let $\epsilon>0$. There exists $\delta>0$ such that for $x_1,\ldots,x_\kb\in E\cap B(x_0,\delta)$ there exist $P_i,P_i'\in H^Q(x_i)$ ($1\le i\le \kb$) satisfying
    \begin{equation*}
        |\da (P_i-P_j)((x_i)_Q)|\le\epsilon/R|(x_i)_Q-(x_j)_Q|^{m-|\alpha|}
        \text{for all }|\alpha|\le m, 0\le i,j,\le \kb
    \end{equation*}
    and
    \begin{equation*}
        |\da (P'_i-P'_j)((x_i)_Q)|\le\epsilon/R|(x_i)_Q-(x_j)_Q|^{m-|\alpha|}
        \text{for all }|\alpha|\le m, 0\le i,j,\le \kb,
    \end{equation*}
where $R=R(\lambda)$ is to be determined.

    Define $P_i''=\lambda P_i+(1-\lambda)P_i'$ for $0\le i\le \kb$. Then,
    \begin{align*}
        |\da (P''_i-P''_j)((x_i)_Q)|&\le |\lambda|\cdot |\da (P_i-P_j)((x_i)_Q)|+|1-\lambda|\cdot |\da (P'_i-P'_j)((x_i)_Q)|\\
        &\le (|\lambda|+|1-\lambda|)\epsilon/R|(x_i)_Q-(x_j)_Q|^{m-|\alpha|}\\
        &=\epsilon|(x_i)_Q-(x_j)_Q|^{m-|\alpha|}
    \end{align*}
    for all $|\alpha|\le m, 0\le i,j,\le \kb$ by setting $R=(|\lambda|+|1-\lambda|)$. Thus, $P_0''\in \tilde{H}^Q(x_0)$.
\end{proof}

To address the computability of Glaeser refinement for M-bundles, fix $x_0\in E, Q\in O(n), P_0\in H^Q(x_0)$ and for $x_1,\ldots,x_\kb\in E$ define
\begin{equation*}
    \mathcal{Q}(P_1,\ldots,P_{\kb};x_1,\ldots,x_{\kb})=\sum_{\substack{0\le i,j\le \kb\\ x_i\neq x_j}}\sum_{|\alpha|\le m} \sum_{1\le l\le n-d}\left[\frac{\da(P^l_i-P^l_j)((x_i)_Q)}{|(x_i)_Q-(x_j)_Q|^{m-|\alpha|}}\right]^2.
\end{equation*}

As before, for fixed $x_0,\ldots,x_\kb$, $\mathcal{Q}(P_1,\ldots,P_{\kb};x_1,\ldots,x_{\kb})$ is a positive semidefinite quadratic form in $P_1,\ldots,P_{\kb}$. Since each $H^Q(x_i)$ is an affine space,
\begin{equation*}
    \mathcal{Q}_{\min}(x_1,\ldots,x_{\kb})=\min\{\mathcal{Q}(P_1,\ldots,P_{\kb};x_1,\ldots,x_{\kb}):P_i\in H^Q(x_i), 1\le i\le \kb\}
\end{equation*}
may be computed with linear algebra. We see that $P_0\in \tilde{H}^Q(x_0)$ if and only if $\mathcal{Q}_{\min}(x_1,\ldots,x_{\kb})\to 0$ as $(x_1,\ldots,x_{\kb})\to (x_0,\ldots,x_0)$ in $E^{\kb}$. As in the case of bundles, the Glaeser refinement of M-bundles may also be computed using linear algebra and a limit. While this is done here independently for an infinite number of $Q\in O(n)$, Theorem \ref{thm:finite number Q's} will show this may be reduced to a finite number for the purposes of determining the existence of a manifold extension.

\begin{lemma}[Termination Lemma]\label{lemma:termination}
Let $E\subset\R^n$ be compact and $(H_0(x))_{x\in E}$ be an M-bundle with iterated Glaeser refinements $(H_k(x))_{x\in E}$, $k\ge1$. There exists $l^*=l(n,m,d)$ such that $H_k(x)=H_{l^*}(x)$ for all $k\ge l^*$ and $x\in E$.
\end{lemma}

\begin{proof}
Since each $H_{k+1}^Q(x)$ depends only on $H^Q_k(y)$ for $y$ near $x$ (and not the whole of $H_k(y)$), it suffices to show that $H^Q_k(x)=H^Q_{l^*}(x)$ for all $k\ge l^*,x\in E$, and $Q\in O(n)$.

We observe that Glaeser refinement has the following three properties:
\begin{enumerate}
    \item If $(H^Q(x))_{x\in E}$ and $(\hat{H}^Q(x))_{x\in E}$ agree in a neighborhood of a point $x_0\in E$, then so do their Glaeser refinements.
    \item Each $H^Q(x)$ is an affine space, thus has a well-defined dimension. We have
    \begin{equation*}
        \dim \tilde{H}^Q(x)\le \liminf_{y\to x, y\in E}\dim H^Q(y).
    \end{equation*}
    \item If $H^Q(x)=\P^m(\R^d,\R^{n-d})$ for all $x\in E$, then the Glaeser refinement of $(H^Q(x))_{x\in E}$ is itself.
\end{enumerate}

From these three properties, one can arrive at the desired conclusion using a standard argument (see the proof of Lemma 4.1 in \cite{FI20-book}).

\end{proof}

Having proven termination of Glaeser refinement for general M-bundles, for the remainder of the paper fix $(H_0(x))_{x\in E}$ as in \eqref{eq:def H_0}, using $(H_1(x))_{x\in E}, (H_2(x))_{x\in E}$, and so on to denote its Glaeser refinements.

As mentioned in the Introduction, we may naively hope to prove that there exists a manifold extension if and only if the Glaeser refinement terminates in a nontrivial M-bundle, but this statement is false. We now introduce the ideas needed to formalize the topological condition.

Given a compact set $E\subset\R^n$, and $K(x)\subset \Gr(n,d)$ for $x\in E$, we say $(K(x))_{x\in E}$ is a \textit{Gr-bundle}. A \textit{section} of such a Gr-bundle is a continuous function $f:E\to \Gr(n,d)$ such that $f(x)\in K(x)$ for all $x\in E$. To determine continuity, we use the topology on $\Gr(n,d)$ induced by the metric in \eqref{eq:Grass dist}.

Let $Q\in O(n)$ and $A$ be an $(n-d)\times d$ matrix. Define
\begin{equation*}
    Q\oplus A:=\Col\left(Q\begin{bmatrix}
        I_d\\ A
    \end{bmatrix}\right)\in \Gr(n,d).
\end{equation*}
Note that for $x\in\R^n$ and $F\in C^m(\R^d,\R^{n-d})$, $Q\oplus \nabla F(x_Q)$ is parallel to the tangent plane of $G_Q(F)$ at $x$; in particular, by \eqref{eq:tangent space of graph} $T_xG_Q(F)=x+Q\oplus \nabla F(x_Q)$. That is,
\begin{equation*}
    p(T_xG_Q(F))=Q\oplus \nabla F(x_Q),
\end{equation*}
where $p$ is as in \eqref{eq:natural projection}.

Given an M-bundle $\mathcal{H}=(H(x))_{x\in E}$, define
\begin{equation*}
    \hat{H}(x)=\{W\in \Gr(n,d):\text{ there exists } (Q,P)\in H(x)\text{ such that } Q\oplus \nabla P(x_Q)=W \}
\end{equation*}
and define a Gr-bundle $\widehat{\mathcal{H}}=(\hat{H}(x))_{x\in E}$. This Gr-bundle is meant to represent the collection of possible tangent spaces determined by the jets in the M-bundle $\mathcal{H}$. Note that the $\oplus$ operation used to define $\widehat{\mathcal{H}}$ is computable with just linear algebra.

\begin{lemma}\label{lemma:forward direction}
    Let $E\subset\R^n$ be compact and suppose $E\subset\M$, where $\M$ is either a $C^m$, $d$-dimensional manifold or a $C^m$, $d$-dimensional compact manifold with boundary. Then,
    \begin{enumerate}
        \item The M-bundle $\mathcal{H}_*=(H_{l^*}(x))_{x\in E}$ is nontrivial and
        \item The Gr-bundle $\widehat{\mathcal{H}_*}$ has a section.
    \end{enumerate}
    Furthermore, if $x\in E$ and $Q\in O(n)$ such that
\begin{equation*}
    \mathcal{M}\cap B(x,\delta)=G_Q(F)\cap B(x,\delta)
\end{equation*}
for some $\delta>0$ and $F\in C^m(\R^d,\R^{n-d})$, then $(Q,J_xF)\in H_{l^*}(x)$.
\end{lemma}

\begin{proof}
    Given $x_0\in E$, choose $Q$, $\delta_0>0$, and $F\in C^m(\R^d,\R^{n-d})$ such that
    \begin{equation*}
        \M\cap B(x_0,\delta_0)=G_Q(F)\cap B(x_0,\delta_0),
    \end{equation*}
that is,
    \begin{equation}\label{eq:new ref!}
        y_Q^\perp=F(y_Q)
    \end{equation}
    for all $y\in \M\cap B(x_0,\delta_0)$.

    Fix $\epsilon>0$. By Taylor's theorem applied to $F$, there exists $0<\delta<\delta_0$ such that $s,t\in B((x_0)_Q,\delta)$ implies
\begin{equation}\label{eq:Taylor applied to F}
    |\da (J_sF-J_tF)(s)|\le\epsilon|s-t|^{m-|\alpha|}
        \text{for all }|\alpha|\le m, 0\le i,j,\le \kb.
\end{equation}

Given $x_1,...,x_{\kb}\in E\cap B(x_0,\delta)$, let
    \begin{equation}\label{eq:pick P_i on mfld}
        P_i=J_{(x_i)_Q}(F((x_i)_Q),
    \end{equation}
for $0\le i\le \kb$, noting that $P_i((x_i)_Q)=(x_i)_Q^\perp$ by \eqref{eq:new ref!}, thus $(Q,P_i)\in H_0(x_i)$. Thus, as $(x_1)_Q,...,(x_{\kb})_Q\in E\cap B((x_0)_Q,\delta)$, \eqref{eq:Taylor applied to F} implies $(Q,P_0)\in H_1(x_0)$ by definition of Glaeser refinement.

By similar argument (noting that $\M$ is also locally the graph of a $C^m$ function in $Q$ coordinates at each $y\in E\cap B(x_0,\delta_0)$), we see that $(Q,P_y)\in H_1(y)$ for all $y\in E\cap B(x_0,\delta_0)$, where 
\begin{equation}\label{eq:choose P_y}
    P_y=J_{y_Q}F(y_Q).
\end{equation}

By the above reasoning with $H_k$ and $H_{k+1}$ in place of $H_0$ and $H_1$, respectively, $(Q,P_y)\in H_k(y)$ for all $k\ge0$. In particular, $(Q,P_0)\in H_{l^*}(x_0)$ and since $x_0\in E$ was arbitrary, the M-bundle $(H_{l^*}(x))_{x\in E}$ is nontrivial. This establishes the first point.


As before, let $x\in E$, choose $Q$, $\delta>0$, and $F\in C^m(\R^d,\R^{n-d})$ such that
    \begin{equation*}
        \M\cap B(x,\delta)=G_Q(F)\cap B(x,\delta),
    \end{equation*}
    
By the prior argument, $(Q,J_{x_Q}F)\in H_{l^*}(x)$. By \eqref{eq:tangent space of graph},
\begin{equation*}
    T_x\M=x + Q\oplus \nabla F(x_Q)= x + Q\oplus \nabla J_{x_Q}F(x_Q).
\end{equation*}
thus, the element of $\Gr(n,d)$ parallel to $T_x\M$ lies in $\hat{H}_{l^*}(x)$. Since $T_x\M$ varies continuously in $x$, and the natural projection $p$ in \eqref{eq:natural projection} is continuous,
\begin{equation*}
    p(T_x\M)=Q\oplus \nabla J_{x_Q}F(x_Q)\in \hat{H}(x)
\end{equation*}
is a section of $(\hat{H}_{l^*}(x))_{x\in E}$.

\end{proof}

At this point, we have shown that if $E$ has a manifold extension, then repeated Glaeser refinements of $(H_0(x))_{x\in E}$ terminate in a nontrivial Glaeser stable M-bundle $(H_{l^*}(x))_{x\in E}$ such that $(\hat{H}_{l^*}(x))_{x\in E}$ has a section. To complete our answer to Problem \ref{prob:manifold WEP}, we show that any such M-bundle implies a manifold extension of $E$.

The proof leads to a slightly stronger result; to state it, we require some more terminology.

An M-bundle is \textit{proper} if $P(x_Q)=x_Q^\perp$ for all $(Q,P)\in H(x)$ and all $x\in E$, that is, if $H(x)\subset H_0(x)$ for all $x\in E$. If an M-bundle is proper, so are all its Glaeser refinements.

Let $(H(x))_{x\in E}$ be a proper M-bundle. Let $x\in \R^n$, $P\in H^Q(x)$, and choose $F\in C^m(\R^d,\R^{n-d})$ such that $J_{x_Q}F=P$. If $R\in O(n)$ such that $G_Q(F)$ is locally the graph of a $C^m$ function $F'$ in $R$ coordinates in a neighborhood of $x$, we define the \textit{realization of the jet} $(Q,P)$ at $x$ in the $R$ coordinate system as
\begin{equation*}
    J_{x,R}(Q,P)=(R,J_{x_R}F')\in O(n)\times \P^m(\R^d,\R^{n-d}).
\end{equation*}
This object is shown to be well-defined (independent of choice of $F$) in Lemma \ref{lemma:jet realization well-defined}.

We say an M-bundle $(H(x))_{x\in E}$ is \textit{consistent} if, whenever $x\in E$, $(Q,P)\in H(x)$, and $G_Q(P)$ is locally the graph of a $C^m$ function in $R$ coordinates in a neighborhood of $x$ for $R\in O(n)$, then $J_{x,R}(Q,P)\in H(x)$. We will show that $(H_{l^*}(x))_{x\in E}$ is consistent in Lemma \ref{lemma:diff Q okay}.

\begin{remark}\label{rmk:only remark (so far)}
    Observe that if $\M$ is locally the graph of a $C^m$ function $F$ in $Q$ coordinates and a $C^m$ function $F'$ in $R$ coordinates, then
    \begin{equation*}
        p(T_x\M)=Q\oplus \nabla F(x_Q)
    \end{equation*}
    and
    \begin{equation*}
        p(T_x\M)=R\oplus \nabla F'(x_R).
    \end{equation*}
    Thus, when $J_{x,R}(Q,P)=(R,P')$, we have
    \begin{equation*}
        Q\oplus \nabla P(x_Q)=R\oplus \nabla P'(x_R).
    \end{equation*}
\end{remark}

The following theorem is the main result of this paper.

\begin{theorem}\label{thm:main}
Fix positive integers $m$ and $d\le n$ and let $E\subset \R^n$ be compact. Let $\mathcal{H}=(H(x))_{x\in E}$ be a nontrivial, Glaeser-stable, proper, consistent M-bundle.

If the Gr-bundle $\widehat{\mathcal{H}}=(\hat{H}(x))_{x\in E}$ has a section, then there exists a $C^m$, $d$-dimensional, compact manifold with boundary $\M$ such that $E\subset\mathcal{M}\subset\R^n$.

Furthermore, if for $x_0\in E$, $W_0\in \hat{H}(x_0)$ and $\widehat{\mathcal{H}}$ has a section $f$ such that $f(x_0)=W_0$, then for all $(Q,P)\in H(x)$ satisfying $Q\oplus \nabla P((x_0)_Q)=W_0$, one may take $\mathcal{M}$ such that
\begin{equation}\label{eq:REF_tilde}
    \mathcal{M}\cap B(x_0,\delta)=G_Q(F)\cap B(x_0,\delta)
\end{equation}
for some $\delta>0$ and $F\in C^m(\R^d,\R^{n-d})$ satisfying $J_{(x_0)_Q}F=P$.
\end{theorem}

In particular, Theorem \ref{thm:main} applies in the case $(H(x))_{x\in E}=(H_{l^*}(x))_{x\in E}$, which is proper by definition, consistent by Lemma \ref{lemma:diff Q okay}, and Glaeser-stable by Lemma \ref{lemma:termination}.

We can do better in the case $m=1$.

By a \textit{section} of an M-bundle $(H(x))_{x\in E}$, we mean a manifold $\mathcal{M}\supset E$ such that the following holds. Let $x\in E$ and choose $Q\in O(n)$ such that
\begin{equation*}
    \mathcal{M}\cap B(x,\delta)=G_Q(F)\cap B(x,\delta)
\end{equation*}
for some $F\in C^m(\R^d,\R^{n-d})$ and $\delta>0$. Then, $(Q,J_{x_Q}F)\in H(x)$.

The above definition suggests we only consider proper M-bundles, so that the section contains the set $E$. It also suggests M-bundles should be consistent since any small neighborhood of a manifold may be viewed as the graph of a function in many different coordinate systems; the definition of a section requires that the resulting jet lie in $H(x)$ for any such coordinate system.

\begin{theorem}\label{thm:m=1 case}
    Let $m=1$. Fix positive integers $d\le n$ and let $E\subset \R^n$. Let $\mathcal{H}=(H(x))_{x\in E}$ be a nontrivial, Glaeser-stable, proper, consistent M-bundle.

If the Gr-bundle $\widehat{\mathcal{H}}=(\hat{H}(x))_{x\in E}$ has a section, then $\mathcal{H}$ has a section.

Furthermore, if $W\in \hat{H}(x)$ and $\widehat{\mathcal{H}}$ has a section $f$ such that $f(x)=W$, then for all $(Q,P)\in H(x)$ satisfying $Q\oplus \nabla P(x_Q)=W$, one may take the section of $\mathcal{H}$ to be a compact manifold with boundary $\mathcal{M}$ such that
\begin{equation*}
    \mathcal{M}\cap B(x,\delta)=G_Q(F)\cap B(x,\delta)
\end{equation*}
for some $\delta>0$ and $F\in C^m(\R^d,\R^{n-d})$ satisfying $J_{x_Q}F=P$.
\end{theorem}

We conjecture Theorem \ref{thm:m=1 case} holds for $m\ge2$; however, this does not appear to follow from our current pasting process in Section \ref{sec:pasting}.

\section{Some Technical Lemmas}\label{sec:technical results}

\begin{lemma}\label{lemma:classify compatbility}
    Let $Q,R\in O(n)$, $x\in\R^n$, and $F\in C^m(\R^d,\R^{n-d})$. Then, there exist $\delta>0$ and $F'\in C^m(\R^d,\R^{n-d})$ such that
    \begin{equation*}
        G_Q(F)\cap B(x,\delta)=G_R(F')\cap B(x,\delta)
    \end{equation*}
    if and only if the top $d\times d$ minor of
    \begin{equation*}
        A:=R^TQ\begin{bmatrix}
        I_d\\ \nabla F(x_Q)
    \end{bmatrix}
    \end{equation*}
    is invertible.
\end{lemma}

In either case, we say $G_Q(F)$ and $R$ are \textit{compatible} near $x$.

\begin{proof}
    The tangent plane $T_x[G_Q(F)]$ is parallel to the plane spanned by the columns of $Q\begin{bmatrix}
        I_d\\ \nabla F(x_Q)
    \end{bmatrix}$. We ask if there exists an $(n-d)\times d$ matrix $B$ such that $R\begin{bmatrix}
            I_d\\ B
        \end{bmatrix}$
    has the same column space as $Q\begin{bmatrix}
        I_d\\ \nabla F(x_Q)
    \end{bmatrix}$, so that $T_x[G_Q(F)]$ may be viewed as an affine function in $R$ coordinates. Equivalently, we ask that
    \begin{equation*}
        \Col(A)=\Col\left(\begin{bmatrix}
            I_d\\ B
        \end{bmatrix}\right).
    \end{equation*}
    By basic properties of Gaussian elimination, this is true if and only if the top $d\times d$ minor of $A$ is invertible.
\end{proof}

\begin{lemma}\label{lemma:jet realization well-defined}
    Let $x\in\R^n, (Q,P)\in H_0(x)$. If $G_Q(P)$ and $R\in O(n)$ are compatible near $x$, then $J_{x,R}(Q,P)$ is well-defined.

    Furthermore,
    \begin{equation}\label{eq:still in H_0}
        J_{x_R}F'(x_R)=x_R^\perp;
    \end{equation}
    thus, $(R,J_{x_R}F')=J_{x,R}(Q,P)\in H_0(x)$.
\end{lemma}

\begin{proof}
Let $F\in C^m(\R^d,\R^{n-d})$ such that $J_{x_Q}F=P$. All $y\in G_Q(F)$ satisfy
\begin{equation}\label{eq:graph of F in Q}
    y_Q^\perp=F(y_Q).
\end{equation}

Since $G_Q(P)$ and $R$ are compatible near $x$, so are $G_Q(F)$ and $R$, as this property depends only on $\nabla F$ by Lemma \ref{lemma:classify compatbility}. Thus, there exist $\delta>0$ and $F'\in C^m(\R^d,\R^{n-d})$ such that $y\in G_Q(F)\cap B(x,\delta)$ satisfy
\begin{equation}\label{eq:graph of F' in R}
    y_R^\perp=F'(y_R).
\end{equation}

In particular,
\begin{equation*}
    x_R^\perp=F'(x_R)
\end{equation*}
since $x\in G_R(F')$. This establishes \eqref{eq:still in H_0}.

By \eqref{eq:def coords},
\begin{equation*}
    \begin{bmatrix}
        y_Q\\ y_Q^\perp
    \end{bmatrix}=Q^TR \begin{bmatrix}
        y_R\\ y_R^\perp
    \end{bmatrix}:=\begin{bmatrix}
        A_1 & A_2\\
        A_3& A_4
    \end{bmatrix}\begin{bmatrix}
        y_R\\ y_R^\perp
    \end{bmatrix},
\end{equation*}
thus \eqref{eq:graph of F in Q} becomes
\begin{equation}\label{eq:for taking derivatives}
    A_3y_R+A_4y_R^\perp=F(A_1y_R+A_2y_R^\perp).
\end{equation}

Let $S(k)$ represent the following statement:
\begin{equation*}
    \nabla^k F'(y_R)=f(Q,R,\nabla F(A_1y_R+A_2y_R^\perp),\ldots,\nabla^k F(A_1y_R+A_2y_R^\perp))
\end{equation*}
for $y_R$ in a neighborhood of $x_R$, where $f$ is infinitely differentiable.

Our goal is to prove $S(k)$ for $1\le k\le m$. We do so by induction on $k$.

Let $k=1$. Taking the derivative of \eqref{eq:for taking derivatives} with respect to $y_R$, one has
\begin{equation*}
    A_3+A_4\frac{\partial y_R^\perp}{\partial y_R}=\nabla F(A_1y_R+A_2y_R^\perp)(A_1+A_2\frac{\partial y_R^\perp}{\partial y_R}).
\end{equation*}

The above is a linear equation in $\nabla F'=\frac{\partial y_R^\perp}{\partial y_R}$, the solution having $C^\infty$ dependence on the parameters in the equation: $A_i$ (meaning $Q$ and $R$) and $\nabla F(A_1y_R+A_2y_R^\perp)$.

Suppose $S(k)$ holds. By a simple application of the multivariate Chain Rule, one sees that the $(k+1)$-st order derivatives of $F'$ depend only on $Q,R$, $\nabla F(A_1y_R+A_2y_R^\perp),\ldots,\nabla^{k+1} F(A_1y_R+A_2y_R^\perp)$, and $\nabla F'(A_1y_R+A_2y_R^\perp)$, which in turn depends only on $Q,R$, and $\nabla F(A_1y_R+A_2y_R^\perp)$ by the base case $k=1$. Thus, we have established $S(k+1)$ and by induction, $S(k)$ holds for all $k$ up to $m$.

In particular, $\nabla F'(x_R),\ldots,\nabla^m F'(x_R)$ depend only on $Q,R$, and $\nabla F(x_Q),\ldots,\nabla^m F(x_Q)$ (which in turn depend only on $P$), proving our claim.

\end{proof}

While we will not make direct use of the following corollary, we include it to shed light on the situation and begin to show that $(Q_1,P_1)\sim_x (Q_2,P_2)$ if $J_{x,Q_2}(Q_1,P_1)=(Q_2,P_2)$ is an equivalence relation on $O(n)\times \P^m(\R^d,\R^{n-d})$.

\begin{corollary}\label{cor:only cor?}
 Fix $x\in\R^n, P\in\P^m(\R^d,\R^{n-d}), Q\in O(n)$ such that $P(x_Q)=x_Q^\perp$. If $G_Q(P)$ and $R\in O(n)$ are compatible near $x$, then $J_{x,Q}(J_{x,R}(Q,P))=(Q,P)$.
\end{corollary}

\begin{proof}
    By Lemma \ref{lemma:jet realization well-defined}, to compute $J_{x,R}(Q,P)$ we may choose $\delta>0$ and $F'\in C^m(\R^d,\R^{n-d})$ such that $G_R(F')\cap B(x,\delta)=G_Q(P)\cap B(x,\delta)$ so $J_{x,R}(Q,P)=(R,J_{x_R}F')$. (Here, we choose $P$ as a $C^m$ function $F$ such that $J_{x_Q}F=P$.)

    To compute $J_{x,Q}(R,J_{x_R}F')$ we may find $\delta'>0$ and $F''\in C^m(\R^d,\R^{n-d})$ such that $G_R(F')\cap B(x,\delta')=G_Q(F'')\cap B(x,\delta')$. However, we already know from our prior computation that taking $F''=P$ and $\delta'=\delta$ works. We observe $J_{x_Q}P=P$ to complete the proof.
\end{proof}

\begin{lemma}[Local Interpolants]\label{lemma:construct local interpolant}
    Let $(H(x))_{x\in E}$ be a proper, Glaeser-stable M-bundle. Let $Q\in O(n)$ and $x_0\in E$. 

If $P\in H^Q(x_0)$, then there exists $\delta>0$ and $F\in C^m(\R^d,\R^{n-d})$ such that $G_Q(F)\supset E\cap B(x_0,\delta)$. Furthermore, in this scenario, one may take $J_{(x_0)_Q}F=P$ and $J_{x_Q}F\in H^Q(x)$ for all $x\in E\cap B(x_0,\delta)$.
    
Conversely, if there exists $\delta>0$ and $F\in C^m(\R^d,\R^{n-d})$ such that $G_Q(F)\supset E\cap B(x_0,\delta)$, then $P\in H_{l^*}^Q(x_0)$.
\end{lemma}

\begin{proof}
Suppose $P\in H^Q(x_0)$. Letting $\epsilon=1$, there exists $\eta>0$ such that whenever $x_1,\ldots,x_\kb\in E\cap B(x_0,\eta)$, there exist $P_1,\ldots,P_\kb\in H^Q(x_i)$ satisfying
\begin{equation}\label{eq:P_i comp for L 4.4}
    |\da(P_i-P_j)((x_i)_Q)|\le |(x_i)_Q-(x_j)_Q|^{m-|\alpha|} \text{ for all } |\alpha|\le m, 0\le i,j\le \kb.
\end{equation}

Let $E'=\{x_Q:x\in E\cap B(x_0,\eta)\}$. We have shown the bundle $(H'(x_Q))_{x_Q\in E'}$, defined by
\begin{equation*}
    H'(x_Q)=\{P\in\P(\R^d,\R^{n-d}):P\in H^Q(x)\},
\end{equation*}
is nontrivial. Note this bundle is well-defined in the sense that if $x,x'\in E\cap B(x_0,\eta)$, then $x_Q=x'_Q$ implies $x=x'$. Suppose for the sake of contradiction that there exist $x,x'\in E\cap B(x_0,\eta)$ such that $x\neq x'$, yet $x_Q=(x')_Q$. However, if $P\in H^Q(x)$ and $P'\in H^Q(x')$, then $P(x_Q)=x_Q^\perp$ and $P'(x'_Q)=(x')_Q^\perp$, contradicting \eqref{eq:P_i comp for L 4.4} in the case $\alpha=0, P_i=P, P_j=P'$.

Our goal is to show that the Glaeser stability of $(H(x))_{x\in E}$ in the M-bundle sense implies the Glaeser stability of $(H'(x))_{x\in E'}$ in the usual sense.

Let $(y_0)_Q\in E', P_0\in H'((y_0)_Q)=H^Q(y_0)$ and $\epsilon>0$. Choose $(y_1)_Q,\ldots,(y_\kb)_Q\in E'\cap B((y_0)_Q,\delta)$ for $\delta>0$ to be determined. By Glaeser stability of $(H(x))_{x\in E}$ (in the M-bundle sense), there exists $0<\delta'<\eta$ such that if $y_1,\ldots,y_\kb\in E\cap B(y_0,\delta')$, there exist $P_1,\ldots,P_\kb\in \P(\R^d,\R^{n-d})$ such that
\begin{equation*}
    P_i\in H^Q(y_i) \text{ for all } 0\le i\le \kb
\end{equation*}
and
\begin{equation*}
    |\da(P_i-P_j)((y_i)_Q)|\le \epsilon|(y_i)_Q-(y_j)_Q|^{m-|\alpha|} \text{ for all } |\alpha|\le m, 0\le i,j\le \kb.
\end{equation*}

Letting $\alpha=0,\eps=1,$ and $j=0$, we have
\begin{equation}\label{eq:4.4.1}
    |(y_i)_Q^\perp-P_0((y_0)_Q)|\le |(y_i)_Q-(y_0)_Q|.
\end{equation}

Since $P_0$ is differentiable, we have
\begin{equation}\label{eq:4.4.2}
    |P_0((y_i)_Q)-(y_0)_Q^\perp|=|P_0((y_i)_Q)-P_0((y_0)_Q)|\le C|(y_i)_Q-(y_0)_Q|
\end{equation}
for $C$ depending only on $P_0$ and $\delta'$.

Thus, by the triangle inequality, \eqref{eq:4.4.1}, and \eqref{eq:4.4.2},
\begin{equation}\label{eq:4.4.3}
    |(y_i)_Q^\perp-(y_0)_Q^\perp|\le |(y_i)_Q^\perp-P_0((y_i)_Q)|+|P_0((y_i)_Q)-(y_0)_Q^\perp|\le C'|(y_i)_Q-(y_0)_Q|
\end{equation}
and by \eqref{eq:4.4.3},
\begin{equation*}
     |y_i-y_0|\le |(y_i)_Q^\perp-(y_0)_Q^\perp|+ |(y_i)_Q-(y_0)_Q|\le C'' |(y_i)_Q-(y_0)_Q|
\end{equation*}
where $C''$ depends only on $P_0$ and $\delta'$.

By taking $\delta=\delta'/C''$, we obtain that for any $y_1,\ldots,y_\kb\in E'\cap B((y_0)_Q,\delta)$, there exist $P_1,\ldots,P_\kb\in \P(\R^d,\R^{n-d})$ such that
\begin{equation*}
    P_i\in H'(y_i) \text{ for all } 0\le i\le \kb
\end{equation*}
and
\begin{equation*}
    |\da(P_i-P_j)((y_i)_Q)|\le \epsilon|(y_i)_Q-(y_j)_Q|^{m-|\alpha|} \text{ for all } |\alpha|\le m, 0\le i,j\le \kb.
\end{equation*}
Thus, $(H'(y))_{y\in E'}$ is Glaeser stable (in the usual sense).

Furthermore, since $(H(x))_{x\in E}$ is proper,
\begin{equation*}
    P(x_Q)=x_Q^\perp \text{ for all } x_Q \in E', P\in H'(x_Q).
\end{equation*}

By Theorem \ref{thm:v-valued}, there exists $F\in C^m(\R^d,\R^{n-d})$ such that $F(x_Q)=x_Q^\perp$ and $J_{x_Q}F\in H'(x_Q)=H^Q(x)$ for all $x_Q\in E'$; thus, $G_Q(F)\supset E\cap B(x_0,\delta)$. Furthermore, we may take $F$ such that $J_{(x_0)_Q}F=P$.

For the reverse direction, suppose there exists $\delta>0$ and $F\in C^m$ such that $G_Q(F)\supset E\cap B(x_0,\delta)$. For all $x\in E\cap B_{\delta}(x_0)$, $(Q,J_{x_Q}F)\in H_0(x)$ since $F(x_Q)=x_Q^\perp$. Since $F\in C^m(\R^d,\R^{n-d})$, the jets $J_{x_Q}F$ are Taylor compatible and the $(Q,J_{x_Q}F)$ survive each round of Glaeser refinement, remaining in $H_{l^*}(x_Q)$, as in the proof of Lemma \ref{lemma:forward direction}. In particular, $(Q,J_{(x_0)_Q}F)\in H_{l^*}(x_0)$.

\end{proof}

\begin{lemma}\label{lemma:diff Q okay}
    Let $E\subset \R^n$ be compact and $(H_{l^*}(x))_{x\in E}$ be as previously defined. Let $x\in E$ and $(Q,P)\in H_{l^*}(x)$. If $G_Q(P)$ and $R\in O(n)$ are compatible near $x$, then $J_{x,R}(Q,P)\in H_{l^*}(x)$. That is, $(H_{l^*}(x))_{x\in E}$ is consistent.
\end{lemma}

\begin{proof}
    The M-bundle $(H_{l_*}(x))_{x\in E}$ is proper by definition and Glaeser-stable by Lemma \ref{lemma:termination}. Thus, by Lemma \ref{lemma:construct local interpolant}, there exists $\delta>0$ and $F\in C^m(\R^d,\R^{n-d})$ such that $G_Q(F)\supset E\cap B(x_0,\delta)$. Since $G_Q(P)$ and $R$ are compatible near $x$, so are $G_Q(F)$ and $R$, as compatibility depends only on $\nabla F$ by Lemma \ref{lemma:classify compatbility}. Thus, there exist $\delta'>0$ and $F'\in C^m(\R^d,\R^{n-d})$ such that $G_Q(F)\cap B(x,\delta')=G_R(F')\cap B(x,\delta')$. By Lemma \ref{lemma:construct local interpolant} (this time in the opposite direction), $J_{x,R}(Q,P)=(R,J_{(x_R)}F')\in H_{l^*}(x)$.
\end{proof}

For $S\subset O(n)$, define
\begin{equation*}
    H_0^S(x)=\{(Q,P)\in S\times P^{m}(\R^d,\R^{n-d}): P(x_Q)=x_Q^\perp\}
\end{equation*}
and denote the fibers of its iterated Glaeser refinements by
\begin{equation*}
    H_k^S(x)=\bigcup_{Q\in S}H_k^Q(x).
\end{equation*}

The following theorem reduces the refinements of $(H^Q(x))_{x\in E}$ required in Theorem \ref{thm:main} from those involving all $Q\in O(n)$ to a finite number of $Q$'s.

\begin{theorem}\label{thm:finite number Q's}
    Fix integers $m,n,d$. There exists a finite $S_0\subset O(n)$ such that the following holds.
    
    Let $E\subset \R^n$ be compact. Then then there exists a $C^m$, $d$-dimensional, compact manifold with boundary $\M$ such that $E\subset\mathcal{M}\subset\R^n$ if and only if
\begin{enumerate}
    \item The M-bundle $\mathcal{H}^{S_0}_*=(H^{S_0}_{l^*}(x))_{x\in E}$ is nontrivial and
    \item The Gr-bundle $\widehat{\mathcal{H}^{S_0}_*}=(\hat{H}^{S_0}_{l^*}(x))_{x\in E}$ has a continuous section.
\end{enumerate}

Furthermore, if for $x\in E$, $\Pi\in \hat{H}^{S_0}_{l^*}(x)$ and $\widehat{(\mathcal{H}^{S_0})^*}$ has a section $f$ such that $f(x)=\Pi$, then for all $(Q,P)\in H^{S_0}_{l^*}(x)$ satisfying $Q\oplus \nabla P(x_Q)=\Pi$, one may take $\mathcal{M}$ to satisfy
\begin{equation*}
    \mathcal{M}\cap B(x,\delta)=G_Q(F)\cap B(x,\delta)
\end{equation*}
for some $\delta>0$ and $F\in C^m(\R^d,\R^{n-d})$ satisfying $J_{x_Q}F=P$.
\end{theorem}

\begin{proof}
Suppose both 1) and 2) above hold.

Since $H_{l^*}^{S_0}(x)\subset H_{l^*}(x)$ for all $x\in E$, the M-bundle $(H_{l^*}(x))_{x\in E}$ is nontrivial. $(H_{l^*}(x))_{x\in E}$ is by definition proper, by Lemma \ref{lemma:diff Q okay} consistent, and by Lemma \ref{lemma:termination} Glaeser-stable. Furthermore, since $\widehat{H^{S_0}_{l^*}}(x)\subset \widehat{H_{l^*}}(x)$ for all $x\in E$, the Gr-bundle $(\widehat{H_{l^*}}(x))_{x\in E}$ has a section. By Theorem \ref{thm:main}, there exists a $d$-dimensional, $C^m$ manifold $\M\supset E$.

Now suppose there exists a $d$-dimensional, $C^m$ compact manifold with boundary $\M\supset E$. We may mostly repeat the argument in the proof of Lemma \ref{lemma:forward direction}; we need only show there exists a finite set $S_0\subset O(n)$, chosen independently of $E$ and $\M$, satisfying the following property. For all $x\in \M$, there exists $Q\in S_0$ such that $T_x\M$ is the graph of a function in $Q$ coordinates.

Choose $S_0$ to be the set of $n\times n$ permutation matrices, that is, those matrices which have a single 1 per row and a single 1 per column with the remainder of the entries being 0. We claim $S_0$ satisfies the desired properties.

At the very least, $\M$ is locally the graph of a $C^m$ function in $Q_0$-coordinates, where the first $d$ columns of $Q_0$ span a $d$-plane parallel to $T_x\M$. By Lemma \ref{lemma:classify compatbility} we want to find $Q\in S_0$ such that the top $d\times d$ minor of $Q^TA$ is invertible, where
\begin{equation*}
    A=Q_0\begin{bmatrix}
    I_d\\ \nabla F(x_Q)
\end{bmatrix}.
\end{equation*}

The matrix $A$ is of rank $d$; therefore, there exist a subset $\{a_{i_1},\ldots,a_{i_d}\}$ of its rows ($i_1,\ldots,i_d$ all distinct) such that the $d\times d$ matrix with rows $a_{i_1},\ldots,a_{i_d}$ is invertible. The result follows by choosing $Q$ to have a 1 in the $i_1$-position of its first column, the $i_2$-position of its second column, up to putting a 1 in the $i_d$ position of the $d$-th column and filling out the rest of $Q$'s columns in any way such that it becomes a permutation matrix.

\end{proof}

In the above proof, we chose a particular choice of $S_0$ for which the conclusion of Theorem \ref{thm:finite number Q's} was particularly easy to verify. There exist smaller sets $S_0$ with the same property; for instance, one could remove permutation matrices from the $S_0$ used above which share the same first $d$ columns with another remaining element. However, we provide no further attempt to determine the smallest such $S_0$.

\section{Constructing Local Interpolants}\label{sec:local interpolants}

We now begin the proof of Theorem \ref{thm:main}. Suppose
\begin{enumerate}
    \item The M-bundle $\H=(H(x))_{x\in E}$ is nontrivial, Glaeser-stable, proper, and consistent;
    \item The Gr-bundle $(\hat{H}(x))_{x\in E}$ has a continuous section $f:E\to \Gr(n,d)$;
    \item For some fixed $x_0\in E$, $f(x_0)=W_0\in \hat{H}(x_0)$. Moreover, $Q_0\oplus \nabla P_0(x_0)=\Pi_0$, where $(Q_0,P_0)\in H(x_0)$.
\end{enumerate}

We define choices $\delta_x\in (0,\infty), R_x\in O(n)$, and $F_x\in C^m(\R^d,\R^{n-d})$ for all $x\in E$ as follows.

Since $f$ is continuous on the compact set $E$, it is uniformly continuous. Choose
\begin{equation}\label{eq:<1}
    0<\delta_f<1
\end{equation}
such that if $y,y'\in E$ and $|y-y'|<2\delta_f$, then $\dist(f(y),f(y'))<c_0/2$, where $c_0>0$ is a small constant to be chosen later. The constant $c_0$ will depend on $d,m,$, and $n$, but never $E$ nor any quantities derived directly from it.

Since $\H$ is proper, Glaeser-stable, and non-empty, by Lemma \ref{lemma:construct local interpolant}, there exists $0<\delta<\delta_f/6$ and $F\in C^m(\R^d,\R^{n-d})$ such that $G_{Q_0}(F)\supset E\cap B(x_0,4\delta)$. Furthermore, $J_{(x_0)_{Q_0}}F=P_0$ and
\begin{equation}\label{eq:A}
    J_{x_{Q_0}}F\in H^{Q_0}(x) \text{ for all } x\in E\cap B(x_0,4\delta).
\end{equation}

Let $R\in O(n)$ such that $W_0= \Pi R$, where $\Pi$ is as in \eqref{eq:define Pi}. Then, $\Pi_0$ and $R$ are compatible near $x_0$ since $W_0=G_{R}(h)$, where $h:\R^d\to\R^{n-d}$ is the zero function. Thus, taking $\delta$ smaller if needed, there exists $F'\in C^m(\R^d,\R^{n-d})$ such that 
\begin{equation}\label{eq:new_cite1}
    G_{R}(F')\cap B(x_0,24\delta)=G_{Q_0}(F)\cap B(x_0,24\delta).
\end{equation}

Since $G_R(F')$ and $G_{Q_0}(F)$ have the same tangent space at $x_0$, 
\begin{equation}\label{eq:87}
    \nabla F'((x_0)_R)=0.
\end{equation}

Again taking $\delta$ smaller if necessary, we may, by the uniform continuity of the tangent space of $G_{R}(F')$ on closed balls, assume that 
\begin{equation}\label{eq:tangent spaces close near x_0}
    \dist(R\oplus \nabla F'(x_R),R\oplus \nabla F'(y_{R}))<c_0/2
\end{equation}
for all $x,y\in B(x_0,\delta)$. In particular,
\begin{equation}\label{eq:tangent spaces close at x_0}
    \dist(R\oplus \nabla F'(x_{R}),W_0)<c_0/2
\end{equation}
for all $x\in B(x_0,\delta)$.

Let $\delta_{x_0}=\delta, F_{x_0}=F'$, and $R_{x_0}=R$, as above.

If $x\in B(x_0,2\delta_{x_0})\setminus\{x_0\}$, choose $Q\in O(n)$ such that $\Pi Q=R_{x_0}\oplus \nabla F_0'(x_{R_{x_0}})$ and $F$ such that 
\begin{equation}\label{eq:new_cite2}
    G_{R_{x_0}}(F_0')\cap B(x,24\delta)=G_Q(F)\cap B(x_0,24\delta)
\end{equation}
for some $0<\delta<\delta_{x_0}/12$. In particular,
\begin{equation}\label{eq:gradient_zero}
    \nabla F(x_Q)=0.
\end{equation}
By \eqref{eq:A} and the consistency of $(H(x))_{x\in E}$, $J_{y_Q}F\in H^Q(y)$ for all $y\in B(x,24\delta)$.

For such $x$, define $\delta_x=\delta, R_x=Q$, and $F_x=F$.

If $x\notin B(x_0,2\delta_0)$, choose $(Q,P)\in H(x)$ such that $Q\oplus \nabla P(x_Q)=f(x)$. By Lemma \ref{lemma:construct local interpolant}, there exists
\begin{equation}\label{eq:B_Iguess}
    0<\delta<\delta_{x_0}/20
\end{equation}
and $F\in C^m(\R^d,\R^{n-d})$ such that $G_{Q}(F)\supset E\cap B(x,\delta)$. Observe that by choice of $\delta$, $B(x,24\delta)$ and $B(x_0,\delta_{x_0})$ are disjoint.

As before, let $R\in O(n)$ such that 
\begin{equation}\label{eq:C}
    f(x)=\Pi R.
\end{equation}
Then, $f(x)$ and $R$ are compatible near $x$, so taking $\delta$ smaller if needed, there exists $F'\in C^m(\R^d,\R^{n-d})$ such that 
\begin{equation}\label{eq:another_one!}
    G_{R}(F')\cap B(x,24\delta)=G_{Q}(F)\cap B(x,24\delta).
\end{equation}
Observe that, as before,  \eqref{eq:C} implies
\begin{equation}\label{eq:last_one}
    \nabla F'((x_0)_R)=0.
\end{equation}

For such $x$, define $\delta_x=\delta, R_x=R$, and $F_x=F'$.

By making $\delta_x$ smaller if necessary, we have, for all $x\in E$,

\begin{enumerate}
    \item By \eqref{eq:new_cite1}, \eqref{eq:new_cite2}, and \eqref{eq:another_one!},
    \begin{equation}\label{eq:property 1.1}
        G_{R_x}(F_x)\cap B(x,24\delta_x)\supset E\cap B(x,24\delta_x)
    \end{equation} 
    \item The bound
    \begin{equation}\label{eq:property 1.2}
        |\nabla F_x(y_{R_x})|\le c_0
    \end{equation}
    holds for all $y_{R_x}\in B(x_{R_x},24\delta_x)\subset\R^d$, hence for all $y\in B(x,24\delta_x)\subset\R^n$. This is allowed because $\nabla F_x$ is continuous and $\nabla F_x(x_{R_x})=0$. (See \eqref{eq:87}, \eqref{eq:gradient_zero}, and \eqref{eq:last_one}.)
    \item If $x,x'\in E$ and $|x-x'|<\delta_f$, then 
    \begin{equation}\label{eq:property 1.3}
        \dist(\Pi R_x,\Pi R_{x'})<c_0
    \end{equation}
    \item If $x\in B(x_0,\delta_{x_0})$, then by \eqref{eq:new_cite2},
    \begin{equation}\label{eq:property 1.4}
        G_{R_x}(F_x)\cap B(x,24\delta_x)=G_{R_{x_0}}(F_{x_0})\cap B(x,24\delta_x).
    \end{equation}
    \item If $x\notin B(x_0,\delta_{x_0})$, then by \eqref{eq:B_Iguess},
    \begin{equation}\label{eq:property 1.5}
        B(x,\delta_x)\cap B(x_0,\delta_x)=\emptyset.
    \end{equation}
\end{enumerate}

The third property above requires some justification. If $x=x_0$ and $x'\notin B(x_0,\delta_{x_0})$ or $x,x'\notin B(x_0,\delta_{x_0})$, then $\Pi R_x=f(x)$ and $\Pi R_{x'}=f(x')$, so \eqref{eq:property 1.3} follows from the definition of $\delta_f$. If $x=x_0$ and $x'\in B(x_0,\delta_{x_0})$ or $x,x'\in B(x_0,\delta_{x_0})$, then it follows from \eqref{eq:tangent spaces close at x_0} and \eqref{eq:tangent spaces close near x_0}, respectively.

The last case is if $x\in B(x_0,\delta_{x_0})$ and $x'\notin B(x_0,\delta_{x_0})$. Here,
\begin{equation*}
    |x'-x_0|\le |x'-x|+|x-x_0|\le \delta_f+\delta_f=2\delta_f;
\end{equation*}
thus, by the definition of $\delta_f$,
\begin{equation*}
    \dist(\Pi R_{x'},W_0)< c_0/2.
\end{equation*}

By \eqref{eq:tangent spaces close at x_0},
\begin{equation*}
    \dist(\Pi R_{x},W_0)< c_0/2.
\end{equation*}

Thus, by the triangle inequality,
\begin{equation*}
    \dist(\Pi R_{x'},\Pi R_x)<c_0,
\end{equation*}
establishing \eqref{eq:property 1.3}.

We add one more condition to the $F_i$. Let
\begin{equation}\label{eq:define E_i}
    E_i:=\{(z)_{Q_i}:z\in E\cap B(z_i,3\delta_0)\}\subset \R^d.
\end{equation}

We require that $F_i$ be $C^\infty$ at $x$ whenever $d(x,E_i)>\delta_0$. This is attainable by a standard argument using a smooth partition of unity to paste together the original $F_i$ with a $C^\infty$ approximating function.

Since $E$ is compact, there exist $x_1,\ldots,x_M\in E$ such that
\begin{equation}\label{eq:x_i subcover}
    E\subset \bigcup_{1\le i\le M} B(x_i,\delta_{x_i}).
\end{equation}

Let $\delta_0=\min_{1\le i\le M}\delta_{x_i}>0$. Observe that  by \eqref{eq:<1},
\begin{equation}\label{eq:B}
    \delta_0<\delta_{x_0}/12<\delta_f/72<1;
\end{equation} else $E\subset B(x_0,\delta_{x_0})$ and we are done. Since $\{B(x,\delta_0):x\in E\}$ is an open cover of $E$, we again apply compactness to determine $y_1,\ldots,y_{M'}\in E$ such that 
\begin{equation*}
    E\subset \bigcup_{1\le j\le M'} B(y_j,\delta_0/3).
\end{equation*}

Next, we use the classical Vitali covering lemma (found in \cite{stein1993harmonic}, for example), stated below:
\begin{lemma}
    Let $B(x_1,r_1),\ldots,B(x_J,r_J) $ be balls in a metric space. Then there exist indices $j_1,\ldots j_m$ such that
    \begin{equation*}
        B(x_1,r_1)\cup \ldots \cup B(x_J,r_J)\subset B(x_{j_1},3r_{j_1})\cup \ldots \cup B(x_{j_m},3r_{j_m})
    \end{equation*}
    and $B(x_{j_1},r_{j_1}), \ldots, B(x_{j_m},r_{j_m})$ are disjoint.
\end{lemma}

Applying the Vitali covering lemma to $\{B(y_j,\delta_0/3): 1\le j\le M'\}$, there exist $z_1,\ldots,z_N\in E$ such that
\begin{equation*}
    E\subset \bigcup_{1\le j\le N} B(z_j,\delta_0)
\end{equation*}
where the $B(z_j,\delta_0/3)$ are disjoint. Here, $\{z_1,\ldots,z_N\}\subset \{y_1,\ldots,y_{M'}\}$ and the $z_j$ are distinct.

From this, we may deduce the existence of a constant $C=C(n)$ such that for all $1\le j\le N$,
\begin{equation*}
    |\{k:B(z_j,12\delta_0)\cap B(z_k,12\delta_0)\neq\emptyset\}|\le C.
\end{equation*}

We summarize the results of this section, replacing $F_j(\cdot)$ by $F_j(\cdot-(z_j)_{Q_j})-(z_j)_{Q_j}^\perp$.

There exist $0<\delta_0<1$, $z_1,\ldots,z_N\in E$, $Q_1,\ldots,Q_N\in O(n)$, and $F_1,\ldots,F_N\in C^m(\R^n,\R^{n-d})$ such that for all $1\le i,j\le N$,
\begin{enumerate}
    \item \begin{equation}\label{eq:property 2.0}
        E\subset\bigcup_{1\le i\le N} B(z_i,\delta_0)
    \end{equation}
    \item \begin{equation}\label{eq:property 2.1}
        F_i(0)=0;
    \end{equation}
    \item \begin{equation}\label{eq:property 2.2}
        z_i+G_{Q_i}(F_i)\supset E\cap B(z_i,24\delta_0);
    \end{equation}
    \item \begin{equation}\label{eq:property 2.3}
        |\nabla F_i(w)|\le c_0
    \end{equation}
    for $|w|\le 12\delta_0$;
    \item \begin{equation}\label{eq:close projections close}
        \dist(\Pi Q_i,\Pi Q_j)<c_0
    \end{equation}
    whenever $B(z_i,12\delta_0)\cap B(z_j,12\delta_0)\neq\emptyset$ by \eqref{eq:property 1.3} and \eqref{eq:B}.
    \item 
    \begin{equation}\label{eq:property 2.5}
    |\{k:B(z_i,12\delta_0)\cap B(z_k,12\delta_0)\neq\emptyset\}|\le C.
\end{equation}
    \item As follows from \eqref{eq:property 1.4}, if $z_i\in B(x_0,6\delta_0)$, then
    \begin{equation}\label{eq:property 2.6}
        [z_i+G_{Q_i}(F_i)]\cap B(z_i,\delta_0)=G_{R_{x_0}}(F_{x_0})\cap B(z_i,\delta_0).
    \end{equation}
    where $x_0, R_{x_0}, F_{x_0}$ are as at the beginning of the section.
    \item
    \begin{equation}\label{eq:property 2.7}
        F_i \text{ is } C^\infty \text{ at distance greater than } \delta_0 \text{ from } E_i,
    \end{equation}
    where $E_i$ is as in \eqref{eq:define E_i}.
\end{enumerate}

\section{Pasting Local Interpolants Together}\label{sec:pasting}

In this section, we write elements $z\in\R^n$ as $(x,y)\in\R^d\times \R^{n-d}$. Let $\Pi_d:\R^n\to\R^d$ send $z\in\R^n$ to its first $d$ coordinates, or equivalently, $x$. We write $B_d(p,r), B_{n-d}(p,r)$, and $B_n(p,r)$ to write the open ball of radius $r$ and center $p$ in $\R^d,\R^{n-d}$, and $\R^n$, respectively. Similarly, $B_d$, $B_{n-d}$, and $B_n$ will denote the unit balls in dimensions $d, n-d$, and $n$, respectively.

For $1\le i\le N$, let $o_i:\R^n\to\R^n$ be the rigid motion defined by
\begin{equation*}
    o_i(z)=Q_iz+z_i,
\end{equation*}
where $Q_i$ and $z_i$ are as in the end of Section \ref{sec:local interpolants}.

Let $\cyl=\delta_0(B_d\times B_{n-d})$, and for $r>0$ let $\cyl^r=r\delta_0(B_d\times B_{n-d})$. Define $\phi:\R^n\to\R$ by letting
\begin{equation*}
    \phi(z)=|y|^2
\end{equation*}
for $z=(x,y)\in \cyl^6$, and $\phi(z)=0$ otherwise. Let $\cyl_i=o_i(\cyl)$ and likewise $\cyl^r_i=o_i(\cyl^r)$ for $r>0$.

Observe that
\begin{equation}\label{eq:balls in cylinders}
B_n(z_i,r\delta_0)\subset\cyl_i^r\subset B_n(z_i,2r\delta_0)
\end{equation}
for all $r>0$.

In the language of this section, \eqref{eq:property 2.1}, \eqref{eq:property 2.2}, and \eqref{eq:balls in cylinders} imply 
\begin{equation}\label{eq:agreement_in_cylinders}
    z_i+G_{Q_i}(F_i)\supset E\cap \cyl_i^{12}.
\end{equation}

\begin{lemma}\label{lemma:it's a packet!}
    For $1\le i\le N$ and let $S_i=\{\cyl^6_{i_1},\ldots,\cyl^6_{i_{|S_i|}}\}$ be the set of cylinders $\cyl_i^6$ which intersect $\cyl^6_i$.

    Then, there exist rigid motions $U_{i_1},\ldots,U_{i_{|S_i|}}$ satisfying $U_{i_j}(z_{i_j})=z_{i_j}$ and translations $T_{i_1},\ldots,T_{i_{|S_i|}}$ such that for $1\le j\le |S_i|$
    \begin{enumerate}
        \item
        \begin{equation}\label{eq:6.1.1}
            T_{i_j}U_{i_j}\cyl^6_{i_j}\text{ is the translation of }\cyl^6_i \text{ by a vector contained in } \R^d;
        \end{equation}
        \item \begin{equation}\label{eq:6.1.2}
            |(I_n-U_{i_j})v|<C(n,d)c_0|v-z_{i_j}| \text{ for all } v\in\R^n,
        \end{equation}
        where $C(n,d)$ is as in Lemma \ref{lemma:Grass bases}; and
        \item \begin{equation}\label{eq:6.1.3}
            |T_{i_j}(0)|<24c_0\delta_0.
        \end{equation}
    \end{enumerate}
\end{lemma}

A collection of cylinders satisfying a conclusion similar to that of the above lemma, as well as a couple of additional properties, would be considered a cylinder packet in \cite{fefferman2016testing}. We omit these additional properties since they are not necessary for our purposes.

\begin{proof}
Without loss of generality, let $Q_i=I_n$ and $z_i=0$.

Let $\Pi_{i_j}$ denote the orthogonal projection onto $Q_{i_j}\R^d$.

    Fix $1\le i\le N$ and $1\le i_j\le |S_i|$. Since $\cyl^6_{i_j}$ intersects $\cyl^6_i$, by \eqref{eq:balls in cylinders}, $B_n(z_i,12\delta_0)$ intersects $B_n(z_{i_j},12\delta_0)$. Thus, by \eqref{eq:close projections close}, $\|\Pi_{i_j}-\Pi_d\|\le c_0$ By Lemma \ref{lemma:Grass bases} and taking $c_0$ sufficiently small if necessary, there exists $R_{i_j}\in O(n)$ such that $R_{i_j}Q_i\R^d=\R^d$ and $\|R_{i_j}-I_n\|\le C(n,d)c_0$. We obtain \eqref{eq:6.1.2} by letting
    \begin{equation*}
        U_{i_j}(z)=R_{i_j}(z-z_{i_j})+z_{i_j}.
    \end{equation*}
    Defining
    \begin{equation*}
        T_{i_j}(z)=z-(z_{i_j})_{I_n}^\perp=z-y_{i_j}
    \end{equation*}
    gives us \eqref{eq:6.1.1}, as
    \begin{align*}
        T_{i_j}U_{i_j}\cyl_{i_j}^6&=T_{i_j}U_{i_j}(Q_{i_j}\cyl_i^6+z_{i_j})\\
        &= T_{i_j}(\cyl_i^6+z_{i_j})\\
        &=\cyl_i^6+x_{i_j}.
    \end{align*}

    Lastly, observe that since $|z_{i_j}-z_i|<24\delta_0$, by \eqref{eq:property 2.2}, $z_{i_j}\in G_{I_n}(F_i)$. By \eqref{eq:property 2.3},
    \begin{equation*}
        |y_{i_j}|=|F_i(x_{i_j})|\le \|\nabla F_i\|\cdot |x_{i_j}|\le 24c_0\delta_0,
    \end{equation*}
   establishing \eqref{eq:6.1.3}.
\end{proof}

Given a $C^k$ manifold $\M$, one may define a function $D:\R^n\to\R$ by
\begin{equation*}
    D(x)=d(x,\M)^2,
\end{equation*}
where $d(x,A)=\inf_{y\in A}d(x,y)$ for any subset $A\subset\R^n$.

The following theorem is adapted from Theorem 13 and Observation 3 in \cite{fefferman2016testing}. It says, roughly, that given a $C^k$ function $F:\R^n\to\R$ satisfying similar properties as $D$ above, there exists a $C^{k-2}$ manifold $\mathcal{N}$ such that $F$ approximates the squared distance function from $\mathcal{N}$. We will also be interested in quantitative bounds on $\mathcal{N}$.

\begin{theorem}\label{thm:testing}
    Suppose $F:B_n\to\R$ is a $C^k$ function satisfying
    \begin{enumerate}
        \item \begin{equation*}
            |\partial^\alpha_{x,y}F(x,y)|\le C_0
        \end{equation*}
        for $(x,y)\in B_n$ and $|\alpha|\le k$.
        \item For $(x,y)\in\R^d\times \R^{n-d}$, we have
        \begin{equation*}
            c_1[|y|^2+\rho^2]\le F(x,y)+\rho^2\le C_1[|y|^2+\rho^2],
        \end{equation*}
        where $0<\rho<c$ and $c$ is a small enough constant depending only on $C_0,c_1,C_1,k,n$.
    \end{enumerate}
    Then, there exist constants $c_2$ and $C$, depending only on $C_0,c_1,C_1,k,n$ such that the following hold:
    \begin{enumerate}
        \item For $z\in B_n(0,c_2)$, let $N(z)$ be the subspace spanned by the eigenvectors of the Hessian $\nabla^2 F(z)$ corresponding to the largest $n-d$ eigenvalues and let $\Pi_{hi}(z):\R^n\to N(z)$ be the orthogonal projection from $\R^n$ onto $N(z)$. Then,
        \begin{equation*}
            |\partial^\alpha \Pi_{hi}(z)|\le C
        \end{equation*}
        for $z\in B_n(0,c_2)$ and $|\alpha|\le k-2$. In particular, $N(z)$ is a $C^{k-2}$ function of $z$.
        \item There exists a map $\Psi:B_d(0,c_2)\to B_{n-d}(0,c_2)$ such that
        \begin{equation*}
            \{z\in B_d(0,c_2)\times B_{n-d}(0,c_2)|\Pi_{hi}(z)\nabla^2 F(z)=0\}=\{(x,\Psi(x))|x\in B_d(0,c_2)\}.
        \end{equation*}
        Furthermore, $\Psi$ satisfies the quantitative bounds
         \begin{equation}\label{eq:control on Psi graph}
            |\Psi(0)|\le C\rho; |\partial^\alpha\Psi|\le C^{|\alpha|}
        \end{equation}
        on $B_d(0,c_2)$ and for $1\le |\alpha|\le k-2$.

        In particular, $\{z\in B_d(0,c_2)\times B_{n-d}(0,c_2)|\Pi_{hi}(z)\nabla^2 F(z)=0\}$ is a $C^{k-2}$ graph.
        \item Any $z\in B_n(0,c_2)$ may be expressed uniquely in the form $(x,\Psi(x))+v$ where $x\in B_d(0,c_2)$ and $v\in \Pi_{hi}(x,\Psi(x))\R^n\cap B_n(0,c_2)$. Furthermore, $x$ and $v$ are both $C^{k-2}$ functions of $z\in B_n(0,c_2)$ with derivatives up to order $k-2$ bounded in absolute value by a constant $C$.
    \end{enumerate}
\end{theorem}

Let $\theta\in C^\infty(\R^d,[0,1])$ be a bump function strictly positive on $B_d(0,1)$, equal to 0 on the complement of $B_d(0,1)$, and satisfying the bounds
\begin{equation}\label{eq:theta def}
    |\da\theta(x)|\le C
\end{equation}
for $x\in\R^d,|\alpha|\le k$.

Define $G:\bigcup_i \cyl^5_i\to\R$ by
    \begin{equation}\label{eq:def F^o}
        G(z)=\frac{\sum_{\cyl_i\ni z}\phi(o_i^{-1}(z))\theta (\Pi_d(o_i^{-1}(z))/(6 \delta_0))}{\sum_{\cyl_i\ni z}\theta (\Pi_d(o_i^{-1}(z))/(6 \delta_0))}
    \end{equation}

Let
\begin{equation*}
    \mathcal{A}_i:=o_i(\delta_0 B_d)\subset \cyl_i;
\end{equation*}
\begin{equation*}
    \mathcal{A}^r_i:=o_i(r\delta_0 B_d)\subset \cyl^r_i
\end{equation*}
for $r>0$;
\begin{equation*}
    \mathcal{A}^r:=\bigcup_i \mathcal{A}^r_i;
\end{equation*}
and, given $R>0$,
\begin{equation*}
    \mathcal{A}^r_R:=\{z\in\R^n:\inf_{\bar{z}\in\mathcal{A}^r}|z-\bar{z}|<R\}.
\end{equation*}

The following is based on Lemma 16 in \cite{fefferman2016testing}.

\begin{lemma}\label{lemma:16 equivalent}
    Let $G$ be as in \eqref{eq:def F^o}. Then, for all $w\in \mathcal{A}^4$, the function $G_w:B_n\to\R$ by
    \begin{equation}\label{eq:def G_w}
        G_w(z)=\frac{G(w+\delta_0Q_i(z))}{\delta_0^2}
    \end{equation}
    satisfies
     \begin{enumerate}
        \item \begin{equation}\label{eq:57.5}
            |\partial^\alpha_{x,y}G_w(x,y)|\le C
        \end{equation}
        for $(x,y)\in B_n$, $|\alpha|\le k$, and a constant $C<\infty$.
        \item For $(x,y)\in B_n$, we have
        \begin{equation}\label{eq:used to have rho's}
            c_1[|y|^2+c_0^2]\le G_w(x,y)+c_0^2\le C_1[|y|^2+c_0^2],
        \end{equation}
        where $c_0$ is independent of any of $C_0,c_1,C_1,k,d,n$.
        \end{enumerate}
\end{lemma}

\begin{proof}
By standard application of the chain rule, to prove the first claim it suffices to show
\begin{equation}\label{eq:derivative bounds on G}
    |\da G(z)|\le C \delta_0^{2-|\alpha|}
\end{equation}
for all $|\alpha|\le m+2$ and $z\in B(w,\delta_0)$, thus $z_0\in \cup_i \cyl_i^5$.

By definition,
\begin{equation*}
    |\da \phi(o_i^{-1}(z))|\le \delta_0^{2-|\alpha|}
\end{equation*}
whenever $\theta (\Pi_d(o_i^{-1}(z'))/(6 \delta_0))$ is not uniformly zero for $z'$ in a neighborhood of $z$ (that is, whenever $\phi(o_i^{-1}(z))$ is relevant for computations of $G$ and $\da G$).

Define, for $z\in \cup_i\cyl^6_i,$
\begin{equation*}
    \theta_i(z)=\frac{\theta (\Pi_d(o_i^{-1}(z))/(6 \delta_0))}{\sum_{\cyl_j\ni z}\theta (\Pi_d(o_j^{-1}(z))/(6 \delta_0))}
\end{equation*}
so that
\begin{equation}\label{eq:rewrite G}
    G(z)=\sum_{\cyl_i}\phi(o_i^{-1}(z))\theta_i(z).
\end{equation}

By repeated applications of the chain rule,
\begin{equation*}
    |\da \theta_i(z)|\le C\delta_0^{-|\alpha|};
\end{equation*}
thus, by \eqref{eq:property 2.5} and multiple applications of the product rule to \eqref{eq:rewrite G}, we obtain \eqref{eq:derivative bounds on G}.

To prove the second claim, let $(x,y)\in B_n$ and $z'=w+\delta_0Q_i(x,y)\in \cyl_i^6$ for some $1\le i \le N$; thus $G_w(x,y)=\delta_0^{-2}G(z')$. Observe that $|y|=d(z',6\mathcal{A}_i)/\delta_0$.

Observe that, for $z\in \cyl_i^6$, $G(z)$ is a convex combination of $d(z,6\mathcal{A}_{i})^2,d(z,6\mathcal{A}_{i_1})^2,\ldots,d(z,6\mathcal{A}_{i_{|S_i|}})^2$, where $S_i=\{\cyl^6_{i_1},\ldots,\cyl^5_{i_{|S_i|}}\}$ is the set of cylinders which intersect $\cyl^6_i$ and $i_1=i$.

It follows from Lemma \ref{lemma:it's a packet!} that for each $1\le j\le |S_i|$, there is a rigid motion of $\cyl^6_{i_j}$ such that the image of $5\mathcal{A}_{i_j}$ is contained in $o_i(\R^d)$; furthermore, by \eqref{eq:6.1.2} and \eqref{eq:6.1.3}, this rigid motion moves no point of $\cyl^6_{i_j}$ more than $Cc_0\delta_0$, where $C=C(n,d)$. Thus, $|d(z,6\mathcal{A}_{i_j})-d(z,6\mathcal{A}_i)|\le Cc_0\delta_0$ for all $1\le j\le |S_i|$.

We now see that, for some $\lambda_1,\ldots,\lambda_{|S_i|}\in [0,1]$ satisfying $\sum_j\lambda_j=1$,
\begin{align*}
    G(z')&= \sum_j \lambda_j d(z',6\mathcal{A}_{i_j})^2\\
    &\le \sum_j\lambda_j [d(z',6\mathcal{A}_i)+Cc_0\delta_0]^2\\
    &\le [\delta_0|y|+Cc_0\delta_0]^2.
\end{align*}
Thus,
\begin{align*}
    G_w(x,y)&\le [|y|+Cc_0]^2\le 4|y|^2+Cc_0^2\\
    G_w(x,y)+c_0^2&\le C(|y|^2+c_0^2),
\end{align*}
establishing the right-hand-side of \eqref{eq:used to have rho's}.

Since $z_0\in\cyl^5_i$ (not just $\cyl_i^6$) and $|S_i|\le C$ by \eqref{eq:property 2.5}, $\lambda_1=\theta_1(z')$ is bounded below by a constant and
\begin{equation*}
    G(z_0)\ge \frac{1}{C'}[2\delta_0|y|]^2,
\end{equation*}
hence
\begin{equation*}
    G_w(x,y)\ge \frac{1}{C''}|y|^2
\end{equation*}
and
\begin{equation*}
    G_w(x,y)+c_0^2\ge \frac{1}{C''}(|y|^2+c_0^2)
\end{equation*}
establishing the left-hand-side of \eqref{eq:used to have rho's}.
    
\end{proof}

The following lemma is based on Lemma 15 in \cite{fefferman2016testing}.
\begin{lemma}\label{lemma:reach bound}
Let
    \begin{equation*}
        \mathcal{M}_{put}:=\{z\in \mathcal{A}^4_{c_2\delta_0/2}|\Pi_{hi}(z)\nabla^2 G(z)=0\}.
    \end{equation*}
    Then, $\M_{put}$ is a $d$-dimensional, $C^{m+2}$ submanifold of $\R^n$ such that the following holds: If $z\in \cyl_i^4$ for some $i$ and $d(z,\M_{put})<c_2\delta_0/4$, then $z$ has a unique closest point in $\M_{put}$.

    Furthermore, if $w\in \mathcal{A}^4$, then $d(z,\M_{put})<c_2\delta_0/2$.
\end{lemma}

\begin{proof}
    We first show that $\M_{put}$ is indeed a manifold. Given $z'\in \M_{put}$, there exists $w\in \mathcal{A}$ such that $z'=w+\delta_0Q_i(z)$ for some $z\in B_n$. Since $\mathcal{A}^4\subset \cup_i\cyl_i^4$, $z'\in \cup_i\cyl_i^5$.
    
    Define $G_w$ as in \eqref{eq:def G_w}. By taking a rigid motion in $\R^n$, we may assume $w=0\in\mathcal{A}_i$ and $Q_i=I_n$. By Lemma \ref{lemma:16 equivalent}, $G_w$ satisfies the hypotheses of Theorem \ref{thm:testing} with $\rho=c_0$ and $k=m+2$; thus, applying \eqref{eq:def G_w} to the conclusion of Theorem \ref{thm:testing}, there exists

    \begin{equation*}
        \Psi:B_d(0,c_2\delta_0)\to B_{n-d}(0,c_2\delta_0)
    \end{equation*}
    such that
    \begin{equation}\label{eq:control on Psi}
            \{z\in B_d(0,c_2\delta_0)\times B_{n-d}(0,c_2\delta_0)|\Pi_{hi}(z)\nabla^2 G(z)=0\}=\{(x,\Psi(x))|x\in B_d(0,c_2\delta_0)\}
        \end{equation}
and
         \begin{equation*}
            |\Psi(0)|\le Cc_0\delta_0^2\le Cc_0\delta_0; |\partial^\alpha\Psi|\le C^{|\alpha|}\delta_0^{2-|\alpha|}
        \end{equation*}
        on $B_d(0,c_2\delta_0)$ and for $1\le |\alpha|\le m$. Here, $C$ and $c_2$ depend only on $c_1,n,d,k$, and the constants $C$ in \eqref{eq:57.5} and \eqref{eq:used to have rho's}.

        We take $c_0$ small enough that $Cc_0<c_2/4$ so that $|\Psi(0)|<c_2\delta_0^2/4<c_2\delta_0/4$ and
        \begin{equation}\label{eq:M close to A}
            d(z,\mathcal{A}^4)<c_2\delta_0/4 \text{ for all }z\in \M_{put}.
        \end{equation}

        Furthermore, any $z\in B_n(0,c_2\delta_0/2)$ may be expressed uniquely in the form $(x,\Psi(x))+v$ where $x\in B_d(0,c_2\delta_0)$ and $v\in \Pi_{hi}(x,\Psi(x))\R^n\cap B_n(0,c_2\delta_0)$. Furthermore, $x$ and $v$ are both $C^{m}$ functions of $z\in B_n(0,c_2\delta_0)$ with derivatives $\da$ up to order $m$ bounded in absolute value by $C\delta_0^{2-|\alpha|}$.
        
        If $z\in \cyl_i^4$ and $d(z,\M_{put})<c_2\delta_0/4$, by \eqref{eq:M close to A}, $d(z,\mathcal{A}^4)<c_2\delta_0/2$. Thus, $z$ has a unique closest point in the graph of $\Psi$, and therefore on $\M_{put}$.

        Since $\mathcal{A}_{c_2\delta_0/2}$ is open, there exists $\eta>0$ such that $B_n(z',\eta)\subset \mathcal{A}_{c_2\delta_0/2}$. We have shown that $\M_{put}$ agrees with the graph of $\Psi$ on $B_n(z',\eta)$; thus, $\M_{put}$ is a manifold.

        Let $w\in \mathcal{A}^4$. By repeating the application of Theorem \ref{thm:testing} to $G_w$, we find that $\M_{put}$ is locally of the form
        \begin{equation*}
            \{w+Q_i(x,\Psi'(x)):x\in B_d(0,c_2\delta_0)\}
        \end{equation*}
        where $|\Psi'(0)|<c_2\delta_0/2$. Thus, $d(w,\M_{put})<c_2\delta_0/2$.

\end{proof}

Let $\M\subset\R^n$ be a $C^{m+2}$, $d$-dimensional manifold and $D$ be an open set containing $\M$. We say a $C^{m+2}$ map $\pi:D\to \M$ is a \textit{disc bundle} if for any $z\in \M$, $\pi(z)=z$ and $D_z:=\pi^{-1}(z)$ is isometric to $B_{n-d}(0,r)$, where $r$ is independent of $z$.

We call a $C^m$ map $s:\M\to D$ a \textit{section} (or \textit{global section}) of $D$ if for any $z\in \M, s(z)\in D_z$. We say $s$ is a \textit{local section} if it is from $U\to D$, where $U\subset \M$ is an open set.

\begin{proof}[Proof of Theorem \ref{thm:main}]
    
We are now ready to complete the proof of Theorem \ref{thm:main}. Let 
\begin{equation*}
    D_{put}=\{z+rv:-1\le r\le 1, z\in \M_{put},v\in N(z), \|v\|\le \frac{1}{2}c_2\delta_0\},
\end{equation*}
where $N(z)$ is the $(n-d)$ dimensional subspace of $\R^n$ consisting of vectors orthogonal to $T_z\M_{put}$. 

Consider the graphs
\begin{equation*}
    G_i:=\{z_i+Q_i(x,F_i(x)):x\in B_d(0,4\delta_0)\}\subset\cyl_i^4, 1\le i\le N.
\end{equation*}
By \eqref{eq:property 2.3} and taking $c_0$ sufficiently small,
\begin{equation}\label{eq:G_i close to A_i}
    d(z,\mathcal{A}^4_i)<c_0\delta_0<c_2\delta_0/4 \text{ for all }z\in G_i, 1\le i\le N.
\end{equation}

By Lemma \ref{lemma:reach bound}, each point of $\mathcal{A}^4_i$ lies within distance $c_2\delta_0/4$ of $\M_{put}$. By the triangle inequality and \eqref{eq:G_i close to A_i},
\begin{equation*}
    d(z,\M_{put})<c_2\delta_0/2 \text{ for all }z\in G_i, 1\le i\le N..
\end{equation*}

Thus, each point of $G_i$ has a unique closest point on $\M_{put}$, so $G_i$ is the graph of a local section $J_i$ of the disc bundle $D_{put}$ on an open subset of $\M_{put}$ by Lemma \ref{lemma:reach bound}.

We construct a smooth partition of unity on $\M_{put}$ as follows. Let
\begin{equation*}
    \tilde{\theta}_j(z)=\begin{cases}
        \theta\left(\frac{\Pi_d(o_j^{-1}z)}{3\delta_0}\right) & z\in \cyl^3_j\\
        0 & \text{else},
    \end{cases}
\end{equation*}
where $\theta$ is as in \eqref{eq:theta def}. Define
\begin{equation}\label{eq:REF'}
    \theta_j(z)=\frac{\tilde{\theta}_j(z)}{\sum_j \tilde{\theta}_j(z)}.
\end{equation}

Define
\begin{equation}\label{eq:pasting of local sections}
    J(x):=\sum_j\theta_j(x)J_i(x).
\end{equation}
Ideally, $J$ would be a global section of $D_{put}$; however, it may not be well-defined near the boundary of $\M_{put}$ (as all $\tilde{\theta}_j$ may be 0). Instead, we take $J$ as a local section of $D_{put}$ on $\M_{put}\cap(\cup_i\cyl_i^3)$ and observe that
\begin{equation*}
    G:=\{x+J(x):x\in \M_{put}\cap(\cup_i\cyl_i^3)\}
\end{equation*}
is a $C^m$ manifold.

Let $z\in E$ and write $z=x+v$, where $x\in \M_{put}$ and $v\in N(x)$. If $J_i$ is defined on $x$, $z\in G_i$, as $G_i$ agrees with $E$ on $B_n(z_i,20\delta_0)\supset \cyl_i^3$. Thus, $J_i(x)=v$, $J(x)=v$, and $z\in G$. In summary, $E\subset G$.

Similarly, since by \eqref{eq:property 2.6} each $G_i$ must agree with $G_{R_{x_0}}(F_{x_0})$ on $B(x_0,\delta_0)$, $G$ must as well, establishing \eqref{eq:REF_tilde}.

Let $\one_i$ denote the indicator function of $\cyl_i^2$ and fix $\varphi\in C^{\infty}(\R^n,\R)$ such that $0\le \varphi\le 1$, $\varphi\equiv1$ on $B(0,\delta_0/100)$ and $\varphi\equiv 0$ on $\R^n\setminus B(0,\delta_0/50)$.

Define $I=\sum_i\one_i$. The convolution $I*\varphi$ satisfies the following properties:
\begin{itemize}
    \item $I*\varphi\in C^\infty$.
    \item $I*\varphi\equiv 1$ on $\cup_i\cyl^{99/50}_i$.
    \item $I*\varphi\equiv0$ on $\R^n\setminus(\cup_i\cyl_i^{101/50})$.
\end{itemize}

Thus, by Sard's theorem, there exists $0<t<1$ such that $R:=\varphi^{-1}([t,1])$ satisfies the following properties:
\begin{itemize}
    \item $\cup_i \cyl^2_i\subset R\subset \cup_i\cyl_i^{101/50}$.
    \item $R$ has $C^\infty$ boundary $\d R$.
\end{itemize}

To see the first point, observe that $\varphi^{-1}((0,1])$ is an open set which contains the closure of $\cup_i\cyl_i^2$. By taking $t$ small enough (which we can since Sard's theorem says the $C^\infty$ boundary holds for almost every $t$), we can ensure this first condition.

We would like to take $\mathcal{M}\cap R$ as the desired $C^m$ manifold with boundary. Clearly, $E\subset \mathcal{M}\cap R$ since $E\subset \cup_i\cyl_i\subset R$. However, we still need to ensure that $\M\cap R$ has a $C^m$ boundary. This will hold if for all $z\in \M\cap \partial R$, $T_z\M+T_z \partial R=\R^n$.

There is no guarantee that this holds in general. However, we may slightly modify the construction of $R$ as follows to ensure that it does.

Observe that $\M_{put}$ is actually a $C^\infty$ manifold; thus, by \eqref{eq:property 2.7} pasting together local sections via a $C^\infty$ partition of unity creates a global section whose graph is a $C^\infty$ manifold at a distance at least $2\delta_0$ from $E$ (in particular at the intersection with $\partial R$).

By Theorem \ref{thm:a.e. is transverse}, there exists arbitrarily small $v\in \R^n$ such that $v+\partial R$ has transverse intersection with $\M$. Thus, by Theorem \ref{thm:transverse intersection is manifold}, for such a $v$, $\M\cap (v+R)$ is the desired $C^m$ manifold with boundary, as its boundary $\M\cap (v+\partial R)$ is a $C^\infty$ manifold.

\end{proof}

\begin{proof}[Proof of Theorem \ref{thm:m=1 case}]
Recall the end of the proof of Theorem \ref{thm:main}, when the local sections of $\M_{put}$ are pasted together (\eqref{eq:pasting of local sections}). We are given a finite number of $F_i\in C^m(\R^d,\R^{n-d})$, each such that $(Q_i,J_{x_{Q_i}}F_i)\in H(x)$ whenever $x\in E\cap (z_i,\delta_0)$. To prove Theorem \ref{thm:m=1 case}, we simply must show that the pasting process preserves this property; the case of $Q\neq Q_i$ is taken care of by the consistency of $(H(x))_{x\in E}$.

Let $z_0\in E$ and consider the pasting of local sections via the smooth partition of unity at $z_0$. By a rigid motion, we may assume that the closest point of $\M_{put}$ to $z_0$ is the origin $0$ in $\R^n$ and that
\begin{equation}\label{eq:tangent space at origin}
    T_{0}\M_{put}=\R^d.
\end{equation}
The section $J(x)$ may be viewed as the graph of a $C^1$ function $F:\R^d\to \R^{n-d}$. Our goal is to compute $\nabla F(0)$ and show the resulting jet lies in $H(z_0)$.

Let $X(z)$ denote the closest point of $\M_{put}$ and $V(z)=z-X(z)$, as in Theorem \ref{thm:testing}. Then we may write
\begin{equation*}
    (w,F(w))=X((w,F(w))+V((w,F(w)).
\end{equation*}
Choose $w_i$ so that 
\begin{equation}\label{eq:def w_i's}
    (w_i,F_i(w_i))=X((w,F(w))+V((w_i,F_i(w_i)).
\end{equation}

From the pasting process,
\begin{equation}\label{eq:sum v_i's}
    V((w,F(w))=\sum_{i=1}^I\Theta_i(w_i)V((w_i,F_i(w_i)),
\end{equation}
where $F_1,\ldots, F_I$ are the (possibly relabeled) local interpolants on the cylinders containing $z_0$, and $(\Theta_i)_{i=1}^I$ is part of a smooth partition of unity
defined by
\begin{equation*}
    \Theta_i(w_i)=\theta_i(X((w_i,F_i(w_i)))),
\end{equation*}
where the $\theta_i$ are as in \eqref{eq:REF'}. Note that by construction, $F(0)+\nabla F_i(0)\cdot x\in H^{I_n}(0)$.

By \eqref{eq:tangent space at origin} and the twice differentiablity of $\M_{put}$, we have, for $z=(x,y)$ near the origin,
\begin{equation}\label{eq:O(thing^2)}
    X(z)=(x,0_{n-d})+O(|z|^2) \text{ and } V(z)=(0_d,y)+O(|z|^2),
\end{equation}

Observe that applying \eqref{eq:O(thing^2)} to \eqref{eq:def w_i's}, we have
\begin{equation*}
    w_i=w+O(|w_i|^2),
\end{equation*}
or
\begin{equation}\label{eq:approx w_i}
    w=w_i+O(|w|^2)
\end{equation}
by the inverse function theorem. Similarly,
\begin{equation*}
    \Theta_i(w_i)=\Theta_i(w)+O(|w|^2).
\end{equation*}

Thus, by \eqref{eq:sum v_i's}, the differentiability of $F$ and the $F_i$'s, and \eqref{eq:approx w_i}
\begin{align*}
    F(w)&=\sum_{i=1}^I\Theta_i(w)F_i(w) + O(|w|^2)+O\left(\sum_{i=1}^I|w_i|^2\right)\\
    &= \sum_{i=1}^I\Theta_i(w)F_i(w) + O(|w|^2).
\end{align*}

By a standard application of the product rule, and the fact that $\sum_{i=1}^I \nabla \Theta_i(0)=0$
\begin{equation}
    \nabla F(0)=\sum_{i=1}^I \Theta_i(0)\nabla F_i(0).
\end{equation}

The jet $F(0)+\nabla F(0)\cdot x\in H^{I_n}(0)$, as $H^{I_n}(0)$ is an affine space, hence closed under convex combinations.
\end{proof}

\section{On Continuous Selections}\label{sec:topologicla condition}

Given a set $A$, we let $2^A$ denote its power set. If $X$ and $Y$ are topological spaces, we say a map $\varphi:X\to 2^Y$ is \textit{lower semi-continuous} when for every open $V\subset Y$, the set $\{x\in X:\varphi(x)\cap V\neq\emptyset\}$ is an open subset of $X$. We call a function $f:X\to Y$ satisfying $f(x)\in \varphi(x)$ for all $x\in X$ a \textit{selection} for $\varphi$ and $g:A\to Y$ satisfying $f(x)\in\varphi(x)$ for all $x\in A\subset X$ a selection for $\varphi$ on $A$.

Given a topological space $X$, and a cover $\mathcal{U}$ of $X$, a \textit{refinement} of $\mathcal{U}$ is a cover $\mathcal{V}$ such that for all $V\in \mathcal{V}$, there exists $U\in \mathcal{U}$ such that $V\subset U$. The \textit{covering dimension} of $X$, denoted $\dim X$, is the smallest number $n$ such that every finite open covering $\mathcal{U}$ of $X$ has a finite, open refinement $\mathcal{V}$ such that each $x\in X$ is in at most $n+1$ elements of $\mathcal{V}$. For a closed subset $A\subset X$, we let the covering dimension of $X\setminus A$ in $X$, denoted $\dim_X(X\setminus A)$, be the smallest $n$ such that $\dim C\le n$ for all closed $C\subset X\setminus A$.

We will use the fact that covering dimension is a topological invariant, that $\dim Y\le\dim X$ for closed $Y\subset X$, that $[0,1]^d$ has covering dimension $d$ (a result known as Lebesgue's covering theorem \cite{Lebesgue}), and the countable sum theorem (see Theorem 3.6 of \cite{charalambous2019dimension}, for instance).

\begin{theorem}[Countable Sum Theorem]\label{thm:countable sum}
    Suppose $X=\cup_{i=1}^\infty X_i$, where $X$ is a normal space and each $X_i$ is a closed subspace of $X$ with $\dim X_i\le n$. Then $\dim X\le n$.
\end{theorem}

We say $X$ is $n$-\textit{connected} if every continuous image of $S^m$ ($m\le n$) in $X$ is contractible. In other words, $\pi_m(X)=\{0\}$ for all $m\le n$, where $\pi_m(X)$ refers to the $m$-th homotopy group of $X$. Given $\mathcal{S}\subset 2^X$, we say $\mathcal{S}$ is equi-$LC^n$ if for every $y\in \cup_{S\in\mathcal{S}}S$ and open $U\ni y$, there exists open $V\ni y$ such that, for every $S\in \mathcal{S}$, every continuous image of an $m$-sphere ($m\le n$) in $S\cap V$ is contractible in $S\cap U$.

The following theorem follows from Theorem 1.2 in \cite{michael1956continuous}.

\begin{theorem}\label{thm:Michael II}
    Let $X$ be a paracompact space, $A\subset X$ closed with $\dim_X(X\setminus A)\le D+1$, $Y$ a complete metric space, $\mathcal{S}\subset 2^Y$ equi-$LC^D$, and $\varphi:X\to 2^Y$ lower semi-continuous.

    If every $S\in\mathcal{S}$ is $D$-connected, then given any selection $f:A\to Y$ for $\varphi$ on a closed $A\subset X$, there exists a selection $F:X \to Y$ for $\varphi$ such that $F(x)=f(x)$ for $x\in A$.
\end{theorem}

\begin{lemma}\label{lemma:dim le d}
    Let $E\subset\R^n$ be compact and let $(H(x))_{x\in E}$ be a nontrivial, Glaeser-stable, proper M-bundle. Then,
    \begin{equation*}
        \dim E\le d.
    \end{equation*}
\end{lemma}

\begin{proof}
Similar to the start of the proof of Theorem \ref{thm:main}, we may apply Lemma \ref{lemma:construct local interpolant} to obtain, for each $x\in X$, $\delta_x>0$, $Q_x\in O(n)$, $F_x\in C^m(\R^d,\R^{n-d})$ such that
\begin{equation*}
    G_{Q_x}(F_x)\cap B_n(x,2\delta_x)\supset E\cap B_n(x,2\delta_x).
\end{equation*}

By compactness of $E$, there exist $x_1,\ldots, x_n\in E$ such that
\begin{equation}\label{eq:Etilde}
    E\subset=\bigcup_{1\le i\le n} G_{Q_i}(F_i)\cap B_n(x_i,\delta_i)\subset \bigcup_{1\le i\le n} \overline{G_{Q_i}(F_i)\cap B_n(x_i,\delta_i)} := \tilde{E},
\end{equation}
where $Q_i=Q_{x_i}, F_i=F_{x_i}$, $\delta_i=\delta_{x_i}$, and $\overline{A}$ denotes the closure of the set $A$.

By Lebesgue's covering theorem and the fact covering dimension is a topological invariant, 
\begin{equation}\label{eq:dim is d}
    \dim [\overline{G_{Q_i}(F_i)\cap B_n(x_i,\delta_i)}]=d.
\end{equation}

By \eqref{eq:dim is d} and Theorem \ref{thm:countable sum} applied to \eqref{eq:Etilde}, $\dim\tilde{E}=d$; thus, $\dim E\le \dim \tilde{E}\le d$.
\end{proof}

\begin{lemma}\label{lemma:lsc}
    Let $E\subset\R^n$ be compact and let $(H(x))_{x\in E}$ be a nontrivial, Glaeser-stable, proper M-bundle. The mapping $\varphi:E\to 2^{\Gr(n,d)}$ by $x\mapsto \hat{H}(x)$ is lower semi-continuous.
\end{lemma}

\begin{proof}
        Let $O\subset \Gr(n,d)$ be open and suppose $O\cap \hat{H}(x)\neq\emptyset$ for some $x\in E$. In particular, let $W_0\in O\cap \hat{H}(x)$.
    
    Since $W_0\in \hat{H}(x)$, there exists $(Q,P)\in H(x)$ such that $W_0=Q\oplus \nabla P(x_Q)$ and choose $\epsilon>0$ such that
    \begin{equation*}
        B_{\Gr}(W_0,\epsilon):=\{W\in \Gr(n,d):\dist(W,W_0)<\epsilon\}\subset O.
    \end{equation*}
    
    By Lemma \ref{lemma:construct local interpolant}, there exist $\delta>0$ and $F\in C^m(\R^d,\R^{n-d})$ such that $G_Q(F)\supset E\cap B(x,\delta)$ and $J_{x_Q}F=P$, thus $W_0=Q\oplus \nabla F(x_Q)$. Furthermore,
    \begin{equation}\label{eq:sec 7 jet containment}
        J_{y_Q}F\in H^Q(y) \text{ for all }y\in B(x,\delta).
    \end{equation}
    The mapping which takes an $(n-d)\times d$ matrix $M$ to $W_M:=Q\oplus M\in \Gr(n,d)$ is continuous in $M$. Thus, there exists $\eta>0$ such that $\|M-\nabla F(x_Q)\|<\eta$ implies $W_M\in B_{\Gr}(W_0,\epsilon)$.

Since $\nabla F(x_Q)$ is continuous in $x_Q$, there exists $0<\delta'<\delta$ such that $|x_Q'-x_Q|<\delta'$ implies $\|\nabla F(x_Q)-\nabla F(x'_Q)\|<\eta$ for $x'\in E$. Observe that $(Q,\nabla F(x'_Q))\in H(x')$ by \eqref{eq:sec 7 jet containment}, so $Q\oplus \nabla F(x'_Q)\in \hat{H}(x')$.

Suppose $|x'-x|<\delta'$ for $x'\in E$. Then, $|x_Q'-x_Q|<\delta$, so by the above reasoning there exists $W'\in \hat{H}(x')$ (namely, $Q\oplus \nabla F(x'_Q)$) such that $W'\in B(W,\epsilon)\subset O$. In particular, $\hat{H}(x')\cap O\neq\emptyset$, proving our claim.   
\end{proof}

In the following lemma, we restrict ourselves to M-bundles arising naturally from manifold extension, with no further constraints. However, one may be able to replicate our analysis for other M-bundles analogously.

\begin{lemma}\label{lemma:sub-Grassmannians}
    Let $E\subset\R^n$ be compact and consider the M-bundle $(H(x))_{x\in E}$, where $H(x)=H_{l^*}(x)$, the $l^*$-th Glaeser refinement of \eqref{eq:def H_0}.
    
    Let $x\in E$. Then, there exist orthonormal $v_1,\ldots,v_l\in \R^n$ such that
\begin{equation}\label{eq:sub-Grass conclusion}
\hat{H}(x)=\{W\in \Gr(n,d):v_1,\ldots,v_l\in W\}.
\end{equation}    
    
     In particular, $\hat{H}(x)$ is a submanifold of $\Gr(n,d)$ which is homeomorphic to $\Gr(n-l,d-l)$.
\end{lemma}

\begin{proof}
For $Q\in O(n)$ define
    \begin{equation}\label{eq:define hat H^Q}
        \hat{H}^Q(x):=\{Q\oplus\nabla P(x_Q):P\in H^Q(x)\}\subset\hat{H}(x).
    \end{equation}

Define $\varphi_Q:\R^{(n-d)\times d}\to \Gr(n,d)$ by
\begin{equation*}
    \varphi_Q(M)=Q\oplus M,
\end{equation*}
the graph of the linear map represented by $M$ in $Q$-coordinates. This is a well known coordinate chart for $\Gr(n,d)$ (see Example 1.36 of \cite{lee2012smooth}, for example).

We claim that
\begin{equation*}
    \hat{H}^Q(x)=\hat{H}(x)\cap \varphi_Q(\R^{(n-d)\times d}).
\end{equation*}

By definition, $\hat{H}^Q(x)\subset \varphi_Q(\R^{(n-d)\times d})$, so by \eqref{eq:define hat H^Q}, $\hat{H}^Q(x)\subset\hat{H}(x)\cap \varphi(\R^{(n-d)\times d})$.

Now let $W\in \hat{H}(x)\cap \varphi_Q(\R^{(n-d)\times d})$ and choose $(Q',P')\in H(x)$ such that $Q'\oplus \nabla P'(x_{Q'})=W$. Since $(H(x))_{x\in E}$ is consistent and $W$ is the graph of a function in $Q$-coordinates, $(Q,P):=J_{x,Q}(Q',P')\in H(x)$ and by Remark \ref{rmk:only remark (so far)}, $Q\oplus \nabla P(x_Q)=W$. Thus, $W\in \hat{H}^Q(x)$, proving our claim.

Fix $Q\in O(n)$ such that $H^Q(x)$ is nonempty. Since $H^Q(x)$ is an affine space and evaluation of the gradient of a polynomial at $x_Q$ is a linear map, $\varphi_Q^{-1}(\hat{H}(x))\subset \R^{(n-d)\times d}$ is an affine space. This is consistent with \eqref{eq:sub-Grass conclusion}, but we will need more.

As in the beginning of the proof of Lemma \ref{lemma:construct local interpolant}, choose $\delta>0$ so that $y\mapsto y_Q$ is injective on $E\cap B_n(x,\delta)$. Let $E'=\{y_Q\in \R^d: y\in B_n(x,\delta)\}$ and $f=(f_1,\ldots,f_{n-d})$, where $f:E'\to \R^{n-d}$ by $y_Q\mapsto y_Q^\perp$.

Observe that Glaeser refinement of M-bundles treats the different components $P_1,\ldots, P_{n-d}$ of polynomials $P=(P_1,\ldots, P_{n-d})\in \p^m(\R^d,\R^{n-d})$ independently.

That is, one may compute $H^Q(x)$ from $H^Q_0(x)$ through separate Glaeser refinements of the M-bundles $(H^j(t))_{t\in E'}$ ($1\le j\le D-d$), where
\begin{equation*}
    (H^j)^Q(y):=\{P\in\p^m(\R^d,\R):P(y_Q)=f_j(y_Q)\}.
\end{equation*}

It is a standard result that the fibers of (usual) bundles defined in this way are translates of ideals in $\p(\R^d,\R)$, and that this property is preserved under Glaeser refinement, provided the fiber remains nonempty (see Lemma 2.1 in \cite{F06}). Furthermore, this ideal depends only on $E'$. This result carries over trivially to the case of M-bundles, thus
\begin{equation}\label{eq:def H^j}
    (H^j)^Q(x)=g_j(x)+I(x),
\end{equation}
where $g_j(x)\in \p(\R^d,\R)$ and $I(x)\subset \p(\R^d,\R)$ is an ideal.

Let
\begin{equation*}
    I'(x)=\{v\in \R^d:\exists P\in I(x), \nabla P(x)=v\}.
\end{equation*}

$I'(x)$ is a subspace of $\R^d$; thus, there exist orthonormal $w_1,\ldots,w_l\in \R^d$ such that
\begin{equation*}
    I'(x)=\{v\in \R^d:\nabla P(x)\cdot w_i=0 \text{ for all } P\in I(x), 1\le i\ le l\}.
\end{equation*}
By rotation in $\R^d$, we may assume $w_1=e_1,\ldots,w_l=e_l$.

Thus,
\begin{equation}\label{eq:use Pi_l}
   I'(x)=\{v\in \R^d:\Pi_lv=0\},
\end{equation}
where $\Pi_l:\R^d\to\R^l$ is the orthogonal projection onto the first $l$ coordinates.

Combining \eqref{eq:def H^j} with \eqref{eq:use Pi_l},
\begin{equation*}
    \{\nabla P:P\in \bar{H}^j(x)\}=\{v\in \R^d:\Pi_l v=(\partial_1 g_j(x),\ldots,\partial_l g_j(x))\}.
\end{equation*}

Putting it all together,
\begin{equation*}
    \hat{H}(x)=\{I_d\oplus M: M\in\R^{(n-d)\times d},\text{ the first } l \text{ columns of } M \text{ are } (\partial_1g_j(x))_{j=1}^{n-d},\ldots (\partial_lg_j(x))_{j=1}^{n-d}\}.
\end{equation*}

Thus, \eqref{eq:sub-Grass conclusion} holds with
\begin{equation*}
    v_i=\begin{bmatrix}
        e_i\\
        \partial_ig_j(x))_{j=1}^{n-d}
    \end{bmatrix}, 1\le i\le l
\end{equation*}
for elements of $\varphi_Q(\R^{(n-d)\times d})$; it remains to extend the conclusion to all of $\hat{H}(x)$.

Let $W_0\in \hat{H}(x)\setminus \varphi_Q(\R^{(n-d)\times d})$. Choose $(Q',P')\in H(x)$ such that $W_0=Q'\oplus \nabla P'(x_{Q'})$.


Choose $W\in \hat{H}(x)\cap \varphi_Q(\R^{(n-d)\times d})\cap \varphi_{Q'}(\R^{(n-d)\times d})$ and an open set $U\subset \varphi_Q(\R^{(n-d)\times d})\cap \varphi_{Q'}(\R^{(n-d)\times d})$. By our prior reasoning,
\begin{equation*}
    \hat{H}^{Q}(x)=\{W\in \varphi_{Q'}(\R^{(n-d)\times d}): v_1,\ldots,v_{l}\in W\}
\end{equation*}
and
\begin{equation*}
    \hat{H}^{Q'}(x)=\{W\in \varphi_{Q'}(\R^{(n-d)\times d}): v_1',\ldots,v_{l'}'\in W\}
\end{equation*}
for $v_1,\ldots,v_{l},v_1',\ldots,v_{l'}'\in \R^n$.

Since both sets must agree on $U$, the span of $v_1,\ldots,v_{l}$ is equal to the span of $v_1',\ldots,v_{l'}'$ and
\begin{equation*}
    \hat{H}^{Q}(x)=\{W\in \varphi_{Q'}(\R^{(n-d)\times d}): v_1,\ldots,v_{l}\in W\}.
\end{equation*}

Since $W_0$ and $Q'$ were arbitrary,
\begin{equation*}
    \hat{H}(x)=\{W\in \Gr(n,d): v_1,\ldots,v_{l}\in W\}.
\end{equation*}

One may take $v_1,\ldots,v_l$ to be orthonormal by the Gram-Schmidt process.





\end{proof}

\begin{lemma}\label{lemma:equi-LC}
    Let $E\subset\R^n$ be compact and let $(H(x))_{x\in E}$ be a nontrivial, Glaeser-stable, proper, consistent M-bundle. Then, $(\hat{H}(x))_{x\in E}$ is equi-$LC^{d-1}$.
\end{lemma}

\begin{proof}
Let $W_0\in \Gr(n,d)$ and let $U\ni W_0$ be an open subset of $\Gr(n,d)$. Choose $Q\in O(n)$ such that $W_0=Q(\R^d)=Q\oplus 0$.

As in the proof of Lemma \ref{lemma:sub-Grassmannians}, we have a coordinate chart $\varphi_Q:\R^{(n-d)\times d}\to \Gr(n,d)$. Pick $V=\varphi_Q(B_{(n-d)\times d}(0,\delta))$ where $\delta>0$ is chosen small enough that $B_{(n-d)\times d}(0,\delta)\subset \varphi_Q^{-1}(U)$.

Let $x\in E$ such that $W_0\in \hat{H}(x)$. We would like to show that every continuous image of $S^m$ ($m\le n$) in $\hat{H}(x)\cap V$ is contractible in $\hat{H}(x)\cap U$. Since $\varphi_Q$ is a homeomorphism and $V\subset U$, it suffices to show that every continuous image of $S^m$ ($m\le n$) in $\varphi_Q^{-1}(\hat{H}(x))\cap B_{(n-d)\times d}(0,\delta))$ is contractible in $\varphi_Q^{-1}(\hat{H}(x))\cap B_{(n-d)\times d}(0,\delta))$. Again returning to the proof of Lemma \ref{lemma:sub-Grassmannians}, we know that $\varphi_Q^{-1}(\hat{H}(x))$ is an affine space; thus the desired property holds for its intersection with a ball.
\end{proof}

Totalling the results of this section we have removed the second condition from Theorem \ref{thm:main} in the case $d=1$ and proven the following: 

\begin{theorem}\label{thm:easy extension}
    Fix positive integers $m,n$ and let $E\subset \R^n$ be compact. Let $\mathcal{H}_{l^*}=(H_{l^*}(x))_{x\in E}$ be the $l^*$-th Glaeser refinement of \eqref{eq:def H_0}.
    
    Then there exists a $C^m$, 1-dimensional manifold $\M$ such that $E\subset\mathcal{M}\subset\R^n$ if and only if $\mathcal{H}_{l^*}$ is nontrivial.

Furthermore, if $\Pi\in \hat{H}_{l^*}(x)$ and $\widehat{\mathcal{H}_{l^*}}$ has a section $f$ such that $f(x)=\Pi$, then for all $(Q,P)\in H_{l^*}(x)$ satisfying $Q\oplus \nabla P(x_Q)=\Pi$, one may take $\mathcal{M}$ such that
\begin{equation*}
    \mathcal{M}\cap B(x,\delta)=G_Q(F)\cap B(x,\delta)
\end{equation*}
for some $\delta>0$ and $F\in C^m(\R^1,\R^{n-1})$ satisfying $J_{x_Q}F=P$.
\end{theorem}

\begin{proof}
    We are in nearly in position to apply Theorem \ref{thm:main}; all that is left is to show that $(\hat{H}_{l^*}(x))_{x\in E}$ has a section. To do this, we would like to apply Theorem \ref{thm:Michael II} in the case where $X=E, A=\emptyset, D=d-1, Y=\Gr(n,d)$, and $\phi:x\mapsto \hat{H}(x)$.

    By Lemma \ref{lemma:sub-Grassmannians}, each $\hat{H}_{l^*}(x)$ is homeomorphic to either $\Gr(n,1)$ (which is in turn homeomorphic to the real projective space, $\mathbb{RP}^{n-1}$) or $\Gr(n,0)$ (a point). In either case, $\hat{H}_{l^*}(x)$ is path-connected so $S^0$ is contractible in it. Therefore, every $\hat{H}_{l^*}(x)$ is 0-connected.

    The remainder of the hypotheses are satisfied directly by Lemmas \ref{lemma:dim le d}, \ref{lemma:lsc}, and \ref{lemma:equi-LC}.

\end{proof}

Theorem \ref{thm:Michael II} only provides sufficient conditions for a selection to exist. There are plenty of circumstances in which there is a manifold extension and the hypotheses of Theorem \ref{thm:Michael II} do not apply, for instance, if $E$ is already a manifold of dimension at least 2. It may be interesting to determine necessary and sufficient conditions for $(\hat{H}(x))_{x\in E}$ to have a section.\\

\noindent \textbf{Acknowledgments.} The author would like to thank Allen Hatcher for helpful communications regarding this work and Shmuel Weinberger for comments on an earlier draft.

\bibliographystyle{plain}
\bibliography{Whitney-bib}

\end{document}